\documentclass[11pt]{article}
\usepackage{calc}
\usepackage{pgf,pgfarrows,pgfnodes,pgfautomata,pgfheaps,pgfshade}
\usepackage[centertags]{amsmath}
\usepackage{latexsym}
\usepackage{amsfonts}
\usepackage{amssymb}
\usepackage{amsthm}
\usepackage{fancyhdr}
\usepackage [dvips]{epsfig}

\usepackage{newlfont}
\usepackage[latin1]{inputenc}
\usepackage[french,english]{babel}
\usepackage{graphicx}
\usepackage{t1enc}
\usepackage[matrix,arrow,frame,curve]{xy}
\theoremstyle{plain}
\newtheorem{lem}{Lemma}[section]
\newtheorem{prop}{Proposition}[section]
\newtheorem{theo}{Theorem}[section]
\newtheorem{definition}{Definition}[section]
\newtheorem{cor}{Corollary}[section]
\newtheorem{rem}{Remark}[section]

\theoremstyle{plain}

\title{Homological properties of certain generalized Jacobian Poisson
structures in dimension 3}

\author{S.R. Tagne Pelap\\
  University of Luxembourg, Mathematics Research Unit\\
  E-mail address: serge.pelap@uni.lu}
\begin{document}
\maketitle
\abstract{
The unimodularity condition for a Poisson structure (ie., a Poisson structure with a trivial modular class) induces a Poincaré duality between its Poisson homology and its Poisson cohomology. Therefore an information about the Poisson homology of this kind Poisson structures induces by duality an information about its Poisson cohomology and vise versa. But this is not longer true in the case of a non trivial modular class. That is the case of Generalized Jacobian Poisson Structures (GJPS). In this paper, we consider certain GJPS in dimension 3 and obtain properties of their Poisson homological groups and their Poisson cohomological groups. More precisely, under some assumptions, we obtain the Poincaré series of these Poisson homological groups and we compute explicitly these Poisson cohomological groups, except the second group which seems more complicated to obtain.
\section*{Introduction}
The Poisson homology has been introduced independently by Brylinski \cite{bry}(as an important tool in the computation of Hochschild homology and the cyclic homology), and by Koszul and Gelfand-Dorfman (inspired by the algebraic approach in the study of bi-hamiltonian structures). There exists a dual notion, the Poisson cohomology, introduced by Lichnerowicz \cite{lic}. This duality is however very subtle. One obtains the Poincaré duality only for a certain class of Poisson structures, called unimodular Poisson structures. The obstructions to the unimodularity are contained in the modular class introduced by Weinstein \cite{weins}. Note also that this kind of duality exists between the Hochschild homology and the Hochschild cohomology for a class of associative algebras called Calabi-Yau, introduced recently by Ginzburg \cite{ginz}, as an algebraic tool of the study of Calabi-Yau manifolds and the mirror symmetry. Hopefully, one can have some relation between the unimodular Poisson structures on a smooth manifold and the Calabi-Yau algebras. This link has been clarified  recently by Dolgushev \cite{dol}. He showed that the deformation quantization of a Poisson algebra $(\mathcal A, \pi_1=\{\cdot, \cdot\})$ on a smooth Poisson manifold  is Calabi-Yau if and only if the formal Poisson bracket $\pi_h=0+\pi_1h+\pi_2h^2+\cdots,$ (where $h$ is the deformation parameter) associated to the star product, is unimodular.  But the problem is that $\pi_1$  can be unimodular while $\pi_h$ is not. This is the case of algebras given by a potential in dimension 3, introduced by Ginzburg \cite{ginz}, which arise as deformations of Jacobian Poisson structures (unimodular), but which are not always Calabi-Yau algebras \cite{boc}. It means that when we deform a unimodular Poisson brackets, the trivial modular class can be deformed to a non trivial one. Then it is natural to think that the obstructions to the unimodularity when one deforms a unimodular Poisson structure should be contained in some cohomological class. The goal of this work is not to study directly this problem, but to study the properties of Poisson homological groups and Poisson cohomological groups of certain Generalized Jacobian Poisson structures in dimension 3 which have a non trivial modular class. This will allow us to understand the role played by the modular class on the "Poincaré duality".\\
The paper is organized as follows: in Section 1, we remain elementary notions about Poisson algebras and Poisson manifolds. In Section 2, we give the definition of the Poisson (co)homology, the modular class and the unimodularity. Section 3 and Section 4 are devoted respectively to some vectorial notations and homological tools we will use later. In the last two sections (5 and 6), we obtain respectively properties of the Poisson homology and the Poisson cohomology of certain GJPS in dimension 3.
\section{Poisson algebras and Poisson manifolds}
Let us start by recalling some preliminary facts concerning Poisson algebras and Poisson manifolds.\\
In the whole paper, $K$ is a field of characteristic zero.\\
Let $\mathcal R$ be a commutative $K$-algebra. $\mathcal R$ is a Poisson algebra if it is a Lie algebra such that the bracket is also a biderivation. This is equivalent to say that $\mathcal R$ is endowed with a $K$-bilinear map $\{\cdot, \cdot\} : \mathcal R\times \mathcal R\longrightarrow\mathcal R$ which satisfies the following properties :\\
$\begin{array}{lr}
   \{a, bc\}=\{a, b\}c+b\{a, c\} & \mbox{(The Leibniz rule)};\\
   \{a, b\}=-\{b, a\}& \mbox{(antisymmetry)};\\
  \{a, \{b, c\}\}+\{b, \{c, a\}\}+\{c, \{a, b\}\}=0& \mbox{(The Jacobi identity)}.\\
\end{array}$\\
where $a, b, c\in\mathcal R.$\\
Then the bracket $\{\cdot, \cdot\}$ is called the Poisson bracket on $\mathcal R$ and one can also say that $\mathcal R$ is endowed with a Poisson structure.
The elements of the center of a Poisson algebra $(\mathcal R, \{\cdot, \cdot\})$ are called the Casimirs and we denote them by $\mathcal C(\mathcal R, \{\cdot, \cdot\})$ or just $\mathcal C(\mathcal R)$, when there is no confusion.  $a\in\mathcal C(\mathcal R, \{\cdot, \cdot\})$ iff $\{a,b\}=0,$ for all $b\in\mathcal R.$\\

A manifold $M$ (smooth, algebraic,...) is said to be a Poisson manifold if its function algebra $\mathcal A$ ($C^\infty(M)$, regular,...) is endowed with a Poisson bracket.\\
Let us consider some specific examples of Poisson algebras. Let
$$q_1=\frac{1}{2}(x_1^2 + x_2^2+x_3^2),$$
$$q_2=\frac{1}{2}(x_0^2+J_1x_1^2 +J_2x_2^2+ J_3x_3^2),$$
be two elements of $\mathbb C[x_0, x_1, x_2, x_3]$ where $J_i\in\mathbb C.$\\
We have a Poisson structure $\pi$ on $\mathbb C^4,$ or on $\mathbb C[x_0, x_1, x_2, x_3],$  such that the bracket between the coordinate functions are defined by (mod $4$):
$$\{x_{0}, x_{i}\}=-J_{jk}x_{j}x_{k};$$
$$\{x_j, x_{k}\}=-x_{0}x_{i},$$
where $J_{jk}=-(J_j-J_k)$ and $(i, j, k)$ forms a cyclic permutation of $(1,2,3).$\\

These algebras, called Sklyanin algebras, were discovered by E.Sklyanin within his
studies of integrable discrete and continuous Landau-Lifschits models by the Quantum
Inverse Scattering Method \cite{sky1},\cite{sky2}. These algebras are intensively studied (\cite{stafnev} \cite{odfe1}, \cite{smsf}, \cite{pelap1}, \cite{pelap2}).
We will denote these algebras by $q_4(\mathcal E),$ where $\mathcal E$ represents the elliptic curve which parameterizes the algebra (via $J_1, J_2, J_3$). We can also think about this curve $\mathcal E$ as being a geometric interpretation of the couple $q_1=0$, $q_2=0$ embedded in $\mathcal CP^3$ (as was observed in Sklyanin's initial paper). The Sklyanin algebras are also called elliptic due to their relation with the elliptic curves. It is interesting to note that almost at the same time where the Sklyanin algebras were introduced, the elliptic algebras with 3 generators were discovered and studied by M. Artin, J. Tate, T.A. Nevins, J.T. Stafford and their
students and collaborators among whom we should mention M. Van den Bergh whose
contribution to the study of the Sklyanin elliptic algebras is very important \cite{ATvdB},\cite{arsc}.
\\

These two Poisson algebras (Sklyanin and Artin-Tate) have two important generalizations: the first from the work of Feigin and Odesskii \cite{odfe1}, $q_{n,k}(\mathcal E),$ where $\mathcal E$ is the elliptic curve and $n, k$ are mutually prime integers. This class of Poisson algebras appears as a quasi-classical limit of quadratic associative algebras denoted by $Q_{n,k}(\mathcal, \eta),$ where $\eta$ is a generic point on the elliptic curve $\mathcal E.$ These are called Sklyanin-Odesskii-Feigin Poisson algebras or just elliptic Poisson algebras. Recently, in a paper \cite{ORT} in collaboration with Giovanni Ortenzi and vladimir Rubtsov, we focussed ourself on the important invariance property of elliptic algebras. Let us remind that if we have an $n-$dimensional vector space $V$ and
fixed a base $v_0,\ldots,v_{n-1}$ of $V$ then the \emph{Heisenberg group of level $n$ in the Schr\"{o}dinger
representation} is the subgroup $H_n\subset GL(V )$ generated by the operators
$$\sigma: v_i \to v_{i-1}; \tau : v_i \to \varepsilon_i(v_i); (\varepsilon_i)^n = 1; 0\leq i \leq n-1.$$
This group has order $n^3$ and is a central extension
$$1 \to {\mathbb U}_n \to H_n \to \mathbb Z_n \times \mathbb Z_n \to 1,$$
where $\mathbb U_n$ is the group of $n-$th roots of unity. This action provides the automorphisms
of the elliptic algebra which are compatible with the grading and defines also an
action on the "quasi-classical" limit of the elliptic algebras - the elliptic Poisson structures $q_{n,k}(\mathcal E)$. We showed that the extension of the Heisenberg group action to more wide family of polynomial quadratic Poisson structures, called $H$-invariant Poisson structures, guarantees also some good and useful features and among
them the \emph{unimodularity} property.

The second generalization is the class Poisson algebras defined as follows: consider $n-2$ polynomials $Q_i$ in $K^n$ with coordinates $x_i$, $i=1,...,n.$ We can define, for any polynomial $\lambda\in K[x_1,...,x_{n}]$, a bilinear differential operation :\\
$$\{\cdot ,\cdot\} : K[x_1,...,x_{n}]\otimes K[x_1,...,x_{n}]\longrightarrow K[x_1,...,x_{n}]$$
by the formula
\begin{equation}\label{q}
\{f,g\}=\lambda\frac{df\wedge dg\wedge dQ_1\wedge...\wedge dQ_{n-2}}{dx_1\wedge dx_2\wedge...\wedge dx_{n}},\ \  \  f,g\in K[x_1,...,x_{n}]
 \end{equation}
This operation gives rise to a Poisson algebra structure on $K[x_1,...,x_{n}].$
The polynomials $Q_i, i=1,...,n-2$ are Casimir functions for the brackets defined by the equation (\ref{q}) and if $Q_1, \cdots, Q_{n-2}$ are functionally independent, then any Casimir is an element of polynomial algebra $K[Q_1, \cdots, Q_{n-2}].$ It has been proven that any Poisson structure on $K^n$, with $n-2$ generic Casimirs $Q_i$, can be written in this form. Every Poisson structure of this form is called a Jacobian Poisson structure (JPS) (~\cite{khi1}, ~\cite{khi2}) for $\lambda=1$ or generalized Jacobian Poisson Structure (GJPS) when $\lambda$ does not belong to $K[Q_1, \cdots, Q_{n-2}].$\\

The main property of JPS and Sklyanin-Odesskii-Feigin Poisson algebras, or more generally quadratic $H$-invariant Poisson algebras, is that they are unimodular, but this is no longer true for GJPS.

\section{Poisson (co)homology-unimodularity}
Let $(\mathcal A, \pi=\{\cdot, \cdot\})$ be a Poisson algebra.
\subsection{Poisson (co)homology complex}
We recall that the $\mathcal A$-module of K\"ahler differentials of $\mathcal A$ is denoted by $\Omega^1(\mathcal A)$ and the graded $\mathcal A$-module $\Omega^p(\mathcal A) :=\bigwedge^p\Omega^1(\mathcal A)$ is the module of all K\"ahler $p$-differential forms.
As a vector space, respectively $\mathcal A$-module, $\Omega^p(\mathcal A)$ is generated by elements of the form $FdF_1\wedge...\wedge dF_p$, respectively $dF_1\wedge...\wedge dF_p$, where $F, F_i\in\mathcal A$, $i=1,...,p$.
We denote by $\Omega^{\bullet}(\mathcal A)=\displaystyle{\oplus_{p\in\mathbb N}}\Omega^p(\mathcal A)$, with the convention that $\Omega^0(\mathcal A)=\mathcal A$, the space of all K\"ahler differential forms.\\
The differential $d : \mathcal A\longrightarrow\Omega^1(\mathcal A)$ extends to a graded $K$-linear map $$d :\Omega^\bullet(\mathcal A)\longrightarrow\Omega^{\bullet+1}(\mathcal A)$$
by setting :
$$d(GdF_1\wedge...\wedge dF_p) :=dG\wedge dF_1\wedge...\wedge dF_p$$
for $G,F_1,...,F_p\in\mathcal A$, where $p\in\mathbb{N}$. It is called the de Rham  differential. It is a graded derivation, of degree $1$, of $(\Omega^\bullet(\mathcal A), \wedge)$, such that $d^2=0$. The resulting complex is called the de Rham complex and its cohomology is the de Rham cohomology of $\mathcal A$.\\

A skew-symmetric $k$-linear map $P\in\mbox{Hom}_K(\wedge^k\mathcal A,\mathcal A)$ is called a skew-symmetric $k$-derivation of $\mathcal A$ with values in $\mathcal A$ if it is a derivation in each of its arguments. The $\mathcal A$-module of skew-symmetric $k$-derivation is denoted by $\mathcal X^k(\mathcal A).$ We define the graded $\mathcal A$-module
$$\mathcal X^\bullet(\mathcal A):=\displaystyle{\bigoplus_{k\in\mathbb N}}\mathcal X^k(\mathcal A)$$
whose elements are called skew-symmetric multi-derivations. By convention, the first term in this sum, $\mathcal X^0$, is $\mathcal A.$\\

Let us now introduce two complexes: a chain and a cochain complex. The first is given by the Poisson boundary operator, also called the Brylinsky or Koszul differential and denoted by $${\partial : \Omega^\bullet(\mathcal A)\longrightarrow\Omega^{\bullet-1}(\mathcal A)}$$
$\partial=[i_{\pi}, d]=i_{\pi}\circ d-d\circ i_{\pi},$
where $i_{\pi}$ is the contraction associated to $\pi.$ More generally, for all $Q\in\mathcal X^q(\mathcal A),$ we can associate the following $\mathcal A$-morphism:
$$i_Q : \Omega^p(\mathcal A)\longrightarrow \Omega^{p-q}(\mathcal A)$$
$$i_Q(df_1\wedge\cdots\wedge df_p)=\displaystyle{\sum_{\sigma\in S_{q, p-q}}}(-1)^{|\sigma|}Q(f_{\sigma(1)},\cdots, f_{\sigma(q)})df_{\sigma(q+1)}\wedge\cdots df_{\sigma(p)},$$
if $p\geq q$ and $i_Q(df_1\wedge\cdots\wedge df_p)=0$ if not.\\
We denote by $S_{p,q}$ the set of all $(p,q)-$shuffles, that is permutations $\sigma$ of the set $\{1,...,p+q\}$ such that $\sigma(1)<...<\sigma(p)$ and $\sigma(p+1)<...<\sigma(p+q), p, q\in\mathbb N $.\\
One can check, by a direct computation, that $\partial_k$ is well-defined and a boundary operator: $\partial_k\circ\partial_{k+1}=0.$\\
The homology of this complex is called the Poisson homology associated to $(\mathcal A, \pi)$ and is denoted by $PH_{\bullet}(\mathcal A, \pi).$
The dual notion of this is the Poisson cohomology: the Poisson coboundary operator associated with $(\mathcal A, \pi)$ is given by
$$\delta : \mathcal X^\bullet(\mathcal A)\longrightarrow\mathcal X^{\bullet+1}(\mathcal A)$$
$\delta(Q)=-[Q,\pi]_S,$ where $[\cdot, \cdot]_S$ is the Schouten bracket:\\
$$[\cdot,\cdot]_S : \mathcal X^p(\mathcal A)\times\mathcal X^q(\mathcal A)\longrightarrow\mathcal X^{p+q-1}(\mathcal A),$$
defined by \\
$\begin{array}{c}
[P, Q]_S(F_1,...,F_{p+q-1})=\displaystyle{\sum_{\sigma\in S_{q,p-1}}}\epsilon(\sigma)P(Q(F_{\sigma(1)},...,F_{\sigma(q)}),F_{\sigma(q+1)},...,F_{\sigma(q+p-1)})\\
\end{array}$\\
$\begin{array}{cccccccccccc}
 & & & & & & & & & & &-(-1)^{(p-1)(q-1)}\displaystyle{\sum_{\sigma\in S_{p,q-1}}}\epsilon(\sigma)Q(P(F_{\sigma(1)},...,F_{\sigma(p)}),F_{\sigma(p+1)},...,F_{\sigma(p+q-1)})\\
\end{array}$\\ \\
for $P\in\mathcal X^p(\mathcal A)$, $Q\in\mathcal X^q(\mathcal A)$, and for $F_1,...,F_{p+q-1}\in\mathcal A$
for $p, q\in\mathbb N.$ By convention, $S_{p,-1}:=\emptyset$ and $S_{-1,q}:=\emptyset$, for $p, q\in\mathbb N.$
One can check, by a direct computation, that $\delta^k$ is well-defined and a coboundary o\-pe\-ra\-tor, $\delta^{k+1}\circ\delta^{k}=0.$\\
The cohomology of this complex is called the Poisson cohomology associated with $(\mathcal A, \pi)$ and denoted by $PH^{\bullet}(\mathcal A, \pi).$
\subsection{Unimodular Poisson structure}
In this part, we consider the affine space of dimension $n,$ $K^n,$ and its algebra of regular functions $\mathcal A=K[x_1,...,x_n]$. Let $\mu=dx_1\wedge\cdots\wedge dx_n.$ More generally, we can consider the smooth algebra of an oriented manifold and fix a volume form $\mu.$ \\
The family of maps $\star^{\mu} : \mathcal X^k(\mathcal A)\longrightarrow\Omega^{n-k}(\mathcal A)$, $\star Q=i_Q(\mu)$ are isomorphisms which give us a Poincaré duality between the multiderivations and the Kälher differential forms.\\ \\
Assume that a Poisson structure, $\pi,$ is given on $\mathcal A.$ It is natural to ask whether we have the same duality between the Poisson cohomology and the Poisson homology. Generally, the answer to this question is negative. Besides, it is easy to see that the answer depends on the Poisson structure we have.\\

Let  $$D_\bullet^{\mu} := (\star^{\mu})^{-1}\circ d\circ\star^{\mu} : \mathcal X^\bullet(\mathcal A)\longrightarrow\mathcal X^{\bullet -1}(\mathcal A)$$
be the pullback of the de Rham differential under the isomorphism $\star^{\mu}.$\\
We can compute $D_2^{\mu}(\pi)$ and we have the following relation between the Poisson boundary and Poisson coboundary:
$$-\star^{\mu}(D_2(\pi)\wedge Q)=\star^{\mu}\delta(Q)+(-1)^{q}\partial(\star^{\mu}(Q)),$$
for all $Q\in\mathcal X^q(\mathcal A).$ Hence if $D_2(\pi)=0,$ then the Poisson coboundary is nothing but the pullback of the Poisson boundary under the isomorphism $\star^{\mu}$ and in this case, we have a Poincaré duality between the Poisson cohomology and the Poisson homology. But more generally, one can prove that the vector field $D_2(\pi)$ is a Poisson $1$-cocycle depending on the choice of the volume form $\mu.$ But its first Poisson cohomological class does not depend on the choice of the volume form $\mu$ and it is called the modular class of the Poisson structure $\pi.$\\
We say that a Poisson bracket $\pi$ on $\mathcal A$ is unimodular if its modular class is trivial. In this case, we get a Poincaré's duality between the Poisson homology and the Poisson cohomology \cite{xu1}. This is the case for JPS ~\cite{khi2} and for quadratic $H$-invariant Poisson structures ~\cite{ORT}. But that is no longer true for GJPS. Our purpose is to study how  a non trivial modular class modifies the "Poincaré duality" in the particular case of GJPS in dimension 3.
\section{Vector notations}
We are going to present some vector notations which we shall use afterwards. Let us consider the following applications and differential operators : \\
$\begin{array}{cccc}
  \times : & \mathcal A^3\times\mathcal A^3 &\longrightarrow & \mathcal A^3 \\
    &\left(\begin{array}{ccc}
                         \overrightarrow{X}=\left(
                             \begin{array}{c}
                               X_1 \\
                               X_2\\
                               X_3 \\
                             \end{array}
                           \right)
          & ,&\overrightarrow{Y}=\left(
                             \begin{array}{c}
                               Y_1 \\
                               Y_2\\
                               Y_3 \\
                             \end{array}
                           \right)\end{array}\right)&                    \longmapsto & \overrightarrow{X}\times\overrightarrow{Y}=\left(
                                                                \begin{array}{c}
                                                                  X_2Y_3-X_3Y_2 \\
                                                                X_3Y_1-X_1Y_3 \\
                                                                  X_1Y_2-X_2Y_1 \\
                                                                \end{array}
                                                              \right)
\end{array}$\\ \\
$\begin{array}{cc}
\begin{array}{cccc}
   \overrightarrow{\nabla} :& \mathcal A &\longrightarrow & \mathcal A^3 \\
    &F&                    \longmapsto &\overrightarrow{\nabla}F=\left(
                                \begin{array}{c}
                                  \frac{\partial F}{\partial x} \\
                                  \frac{\partial F}{\partial y}\\
                                  \frac{\partial F}{\partial z} \\
                                \end{array}
                              \right)\
    \end{array}&\begin{array}{cccc}
  \overrightarrow\nabla\times : & \mathcal A^3 &\longrightarrow & \mathcal A^3 \\
    &\overrightarrow{Y}=\left(\begin{array}{c}
                               Y_1 \\
                               Y_2\\
                               Y_3 \\
                             \end{array}
                           \right)
                           &\longmapsto & \overrightarrow\nabla\times\overrightarrow{Y}=\left(
                                                                \begin{array}{c}
                                                                  \frac{\partial Y_3}{\partial y}-\frac{\partial Y_2}{\partial z} \\
                                                                  \frac{\partial Y_1}{\partial z}-\frac{\partial Y_3}{\partial x}\\
                                                                  \frac{\partial Y_2}{\partial x}-\frac{\partial Y_1}{\partial y}\\
                                                                \end{array}
                                                              \right)
\end{array}
\end{array}$\\ \\
$\begin{array}{cccc}
   \mbox{Div}(\cdot) :& \mathcal A^4 &\longrightarrow & \mathcal A \\
    &\overrightarrow{K}=\left(
                         \begin{array}{c}
                           K_1\\
                           K_2\\
                           K_3 \\
                         \end{array}
                       \right)
    &                    \longmapsto &\mbox{Div}(\overrightarrow{K})=\displaystyle{\sum_{i=1}^{3}}\frac{\partial K_i}{\partial x_i}
    \end{array}$\\ \\
    We denote by $"\cdot"$ the scalar product in $\mathcal A^3.$\\
    By direct computation, we obtain the following properties :
   \begin{prop}\label{pr1}
The previous operators satisfy the following properties :
\begin{enumerate}
    \item[1.]\label{al1} $\overrightarrow{\nabla}\times F\overrightarrow{G}=F\overrightarrow{\nabla}\times\overrightarrow{G} +\overrightarrow{\nabla}F\times\overrightarrow{G},$\ \ $F\in\mathcal A,$ $\overrightarrow{G}\in\mathcal A^3;$
  \item[2.]\label{al2} $Div(F\overrightarrow{G})=\overrightarrow{\nabla}F\cdot\overrightarrow{G}+FDiv(\overrightarrow{G}),$\ \ $F\in\mathcal A,$ $\overrightarrow{G}\in\mathcal A^3;$
\item[3.]\label{al3} $Div(\overrightarrow{F}\times\overrightarrow{G})=\overrightarrow{G}\cdot (\overrightarrow{\nabla}\times\overrightarrow{F})-\overrightarrow{F}\cdot(\overrightarrow{\nabla}\times\overrightarrow{G}),$ $ \overrightarrow{F}, \ \overrightarrow{G}\in\mathcal A^3;$
\item[4.]\label{al4} $\overrightarrow{F}\cdot(\overrightarrow{G}\times\overrightarrow{H})=\overrightarrow{G}\cdot(\overrightarrow{H}\times\overrightarrow{F}),$\ \ $\overrightarrow{F}$ $\overrightarrow{G},$ $\overrightarrow{H}\in\mathcal A^3.$
\end{enumerate}
\end{prop}
According to the definition of K\"ahler differentials and skew-symmetric multi-derivations, we have the following isomorphisms of $\mathcal A$-modules \\
$$\begin{array}{ccc}\Omega^1(\mathcal A)&\stackrel{\cong}{\longrightarrow} &\mathcal A^3\\
F_1dx+F_2dy+F_3dz&\longmapsto &(F_1,F_2, F_3)\\
\end{array}$$
$$\begin{array}{ccc}\Omega^2(\mathcal A)&\stackrel{\cong}{\longrightarrow} &\mathcal A^3\\
F_1dy\wedge dz+F_2dz\wedge dx+F_3dx\wedge dy&\longmapsto &(F_1, F_2, F_3)\\
\end{array}$$
$$\begin{array}{ccc}\Omega^3(\mathcal A)&\stackrel{\cong}{\longrightarrow} &\mathcal A\\
Udx\wedge dy\wedge dz&\longmapsto &U\\
\end{array}$$

$$\begin{array}{ccc}\mathcal X^1(\mathcal A)&\stackrel{\cong}{\longrightarrow} &\mathcal A^3\\
F_1\frac{\partial}{\partial x}+F_2\frac{\partial}{\partial y}+F_3\frac{\partial}{\partial z}&\longmapsto &(F_1,F_2, F_3)\\
\end{array}$$
$$\begin{array}{ccc}\mathcal X^2(\mathcal A)&\stackrel{\cong}{\longrightarrow} &\mathcal A^3\\
F_1\frac{\partial}{\partial y}\wedge \frac{\partial}{\partial z}+F_2\frac{\partial}{\partial z}\wedge \frac{\partial}{\partial x}+F_3\frac{\partial}{\partial x}\wedge \frac{\partial}{\partial y}&\longmapsto &(F_1, F_2, F_3)\\
\end{array}$$
$$\begin{array}{ccc}\mathcal X^3(\mathcal A)&\stackrel{\cong}{\longrightarrow} &\mathcal A\\
U\frac{\partial}{\partial x}\wedge \frac{\partial}{\partial y}\wedge \frac{\partial}{\partial z}&\longmapsto &U\\
\end{array}$$
According to the previous isomorphisms, we can write the de Rham complex in terms of elements of $\mathcal A$, $\mathcal A^4$ and $\mathcal A^6$ :\\
\begin{equation}\label{qa}
\begin{array}{ccccccccccc}
  K&{\hookrightarrow}&\mathcal A & \stackrel{d}{\longrightarrow} & \mathcal A^3 & \stackrel{d}{\longrightarrow} & \mathcal A^3 & \stackrel{d}{\longrightarrow} & \mathcal A & \stackrel{d}{\longrightarrow} &0\\
   & &F & \longmapsto & \overrightarrow{\nabla}F &  &  &  &  &  &\\
   & && & \overrightarrow{F} &\longmapsto & \overrightarrow{\nabla}\times\overrightarrow{F} & &  &  & \\
& &  &  &  &  &\overrightarrow{K} & \longmapsto & \mbox{Div}(\overrightarrow{K}) &
\end{array}
\end{equation}
\begin{prop}{\textnormal{(Poincaré's lemma)}}
The de Rham complex $(\ref{qa})$ of the polynomial algebra $\mathcal A=K[x, y, z]$ is exact.
\end{prop}

\begin{prop}
According to the previous isomorphisms, the Poisson boundary operators associated with the GJPS on algebra $\mathcal A$ given by two polynomials, $\lambda$ and the Casimir $P,$ can be written, using the vector notations, as follows:\\
\begin{equation}\label{hc1}
0\longrightarrow\mathcal A\stackrel{\partial_3}{\longrightarrow}\mathcal A^3\stackrel{\partial_2}{\longrightarrow}\mathcal A^3\stackrel{\partial_1}{\longrightarrow}\mathcal A\longrightarrow 0
\end{equation}\\
where:\\ \\
$\partial_1(\overrightarrow{H})=-\lambda(\overrightarrow\nabla\times\overrightarrow H)\cdot \overrightarrow\nabla P,\  \overrightarrow H\in\mathcal A^3$\\ \\
$\partial_2(\overrightarrow{G})= -\overrightarrow\nabla(\lambda\overrightarrow G\cdot\overrightarrow\nabla P)+ \lambda Div(\overrightarrow G)\overrightarrow\nabla P, \ \overrightarrow G\in\mathcal A^3$\\ \\
$\partial_3(U)=-\overrightarrow\nabla(\lambda U)\times\overrightarrow\nabla P,\ U\in \mathcal A.$
\end{prop}
\begin{proof}
Let us give a proof of the last formula and let $Udx\wedge dy\wedge dz$ be an element of $\Omega^3(\mathcal A).$
We have:
\begin{multline*}
    \partial_3(Udx\wedge dy\wedge dz)= \{U, x\}dy\wedge dz+\{U, y\}dz\wedge dx+\{U, z\}dx\wedge dy+\\
    -Ud\{x, y\}\wedge dz+Ud\{x, z\}\wedge dy-d\{y, z\}\wedge dx.
\end{multline*}
But using our identification, we have $\partial_3(U)= (K_1, K_2, K_3)^t$. $K_1$ is for example the coefficient of $dy\wedge dz.$\\
$\begin{array}{rl}
K_1&=-\lambda\left(\frac{\partial U}{\partial y}\frac{\partial P}{\partial z}-\frac{\partial U}{\partial z}\frac{\partial P}{\partial y}\right)+U\left(\frac{\partial }{\partial z}(\lambda\frac{\partial P}{\partial y})-\frac{\partial }{\partial y}(\lambda\frac{\partial P}{\partial z})\right)\\
&=-\left(\frac{\partial }{\partial z}(\lambda U)\frac{\partial P}{\partial y}-\frac{\partial }{\partial y}(\lambda U)\frac{\partial P}{\partial z} \right)
\end{array}$\\
which is nothing but the first coordinate of $-\overrightarrow\nabla(\lambda U)\times\overrightarrow\nabla P.$
\end{proof}
Let us denote by $\pi$ the GJPS given by $\lambda$ and the Casimir $P.$ Then the Poisson homology takes the following form:\\ \\
$\begin{array}{lll}
PH_0(\mathcal A, \pi)&=&\frac{\mathcal A}{\{-\lambda(\overrightarrow\nabla\times\overrightarrow H)\cdot \overrightarrow\nabla P,\  \overrightarrow H\in\mathcal A^3\}}\\&& \\
PH_1(\mathcal A, \pi)&=&\frac{\{\lambda(\overrightarrow\nabla\times\overrightarrow H)\cdot \overrightarrow\nabla P=0\}}{-\overrightarrow\nabla(\lambda\overrightarrow G\cdot\overrightarrow\nabla P)+ \lambda Div(\overrightarrow G)\overrightarrow\nabla P, \ \overrightarrow G\in\mathcal A^3\}}\\&&\\
PH_2(\mathcal A, \pi)&=&\frac{\{-\overrightarrow\nabla(\lambda\overrightarrow G\cdot\overrightarrow\nabla P)+ \lambda Div(\overrightarrow G)\overrightarrow\nabla P=0\}}{\{-\overrightarrow\nabla(\lambda U)\times\overrightarrow\nabla P,\ U\in \mathcal A\}}\\&&\\
PH_3(\mathcal A, \pi)&=&\{-\overrightarrow\nabla(\lambda U)\times\overrightarrow\nabla P=0\}.\\
\end{array}$
\begin{prop}
According to the previous isomorphisms, the Poisson coboundary operators associated with the GJPS on algebra $\mathcal A$ given by two polynomials, $\lambda$ and the Casimir $P,$ can be written, using the vector notations, as follows:\\
\begin{equation}\label{hc2}
0\longrightarrow\mathcal A\stackrel{\delta^0}{\longrightarrow}\mathcal A^3\stackrel{\delta^1}{\longrightarrow}\mathcal A^3\stackrel{\delta^2}{\longrightarrow}\mathcal A\longrightarrow0
\end{equation}\\
where :\\ \\
$\delta^0(F)=-\lambda\overrightarrow\nabla F\times \overrightarrow\nabla P,\  F\in\mathcal A$\\ \\
$\delta^1(\overrightarrow{F})= -\lambda\overrightarrow\nabla(\overrightarrow F\cdot\overrightarrow\nabla P)+ \left(\lambda Div(\overrightarrow F)-\overrightarrow{F}\cdot\overrightarrow{\nabla}\lambda\right)\overrightarrow\nabla P, \ \overrightarrow F\in\mathcal A^3$\\ \\
$\delta^2(G)=-\lambda\overrightarrow\nabla P\cdot\overrightarrow\nabla\times\overrightarrow{G}-\overrightarrow{G}\cdot(\overrightarrow\nabla \lambda\times \overrightarrow\nabla P),\ \overrightarrow{G}\in \mathcal A^3.$\\
\end{prop}
Then the Poisson cohomology can be rewritten as follows:\\ \\
$\begin{array}{lll}
PH^0(\mathcal A, \pi)&=&\{F\in\mathcal A, \ \ \overrightarrow\nabla F\times \overrightarrow\nabla P=0\}\\&& \\
PH^1(\mathcal A, \pi)&=&\frac{\{ \overrightarrow F\in\mathcal A^3, \ \ -\lambda\overrightarrow\nabla(\overrightarrow F\cdot\overrightarrow\nabla P)+ \left(\lambda Div(\overrightarrow F)-\overrightarrow{F}\cdot\overrightarrow{\nabla}\lambda\right)\overrightarrow\nabla P=0\}}{\{ \lambda\overrightarrow\nabla F\times \overrightarrow\nabla P,\  F\in\mathcal A\}}\\&&\\
PH^2(\mathcal A, \pi)&=&\frac{\{\overrightarrow{G}\in \mathcal A^3, \ \ \lambda\overrightarrow\nabla P\cdot\overrightarrow\nabla\times\overrightarrow{G}+\overrightarrow{G}\times(\overrightarrow\nabla \lambda\times \overrightarrow\nabla P)=0 \}}{\{-\lambda\overrightarrow\nabla(\overrightarrow F\cdot\overrightarrow\nabla P)+ \left(\lambda Div(\overrightarrow F)-\overrightarrow{F}\cdot\overrightarrow{\nabla}\lambda\right)\overrightarrow\nabla P, \ \overrightarrow F\in\mathcal A^3\}}\\&&\\
PH^3(\mathcal A, \pi)&=&\frac{\mathcal A}{\{\lambda\overrightarrow\nabla P\cdot\overrightarrow\nabla\times\overrightarrow{G}+\overrightarrow{G}\times(\overrightarrow\nabla \lambda\times \overrightarrow\nabla P),\ \overrightarrow{G}\in \mathcal A^3\}}\\
\end{array}$

\section{Koszul complex- complete intersection with an isolated singularity in dimension 3}
Let us introduce some homological tools we will need in order to find the Poisson (co)homology of a GJPS given by $\lambda$ and the Casimir $P$ with some additional conditions on $\lambda$ and $P.$ We will suppose that $\lambda$ and $P$ form a regular sequence of weight homogeneous polynomials with 3 variables and define a complete intersection with an isolated singularity.\\
Set $\mathcal A=K[x, y, z]$.
\subsection{Weight homogeneous skew-symmetric multi-derivations}
\begin{definition}
Let $\mathcal V\in\mathcal X^1(\mathcal A)$ and $Q\in\mathcal X^q(\mathcal A)$. By definition, the Lie derivative of $Q$ with respect to $\mathcal V$ is $\mathcal L_{\mathcal V}Q:=[\mathcal V,Q]_S$
\end{definition}
\begin{definition}
A non-zero multi-derivation $P\in\mathcal X^\bullet(\mathcal A)$ is said to be weight homogeneous of degree $r\in\mathbb Z$, if there exists positive integers $\varpi_1, \varpi_2, \varpi_3\in\mathbb N^{\star}$, the weights of the variables $x, y, z$, without a common divisor, such that $$\mathcal L_{\vec{e}_\varpi}(P)=rP$$
where $\mathcal L_{\vec{e}_\varpi}$ is a Lie derivative with respect to the Euler derivation
$$\vec e_\varpi:=\varpi_1x\frac{\partial}{\partial x}+\varpi_2y\frac{\partial}{\partial y}+\varpi_3z\frac{\partial}{\partial z}$$
\end{definition}
The degree of a weight homogeneous multi-derivation $P\in\mathcal X^\bullet(\mathcal A)$ is also denoted by $\varpi(P)\in\mathbb Z.$
By convention, the zero $k$-derivation is weight homogeneous of degree $-\infty$.\\
The Euler derivation $\vec e_\varpi$ is identified to the element $\vec e_\varpi=(\varpi_1x, \varpi_2y,\varpi_3z)^t\in\mathcal A^3$. We denote by $|\varpi|$ the sum of the weights $\varpi_1+\varpi_2+\varpi_3$, so that $|\varpi|=Div(\vec e_\varpi)$.\\
The Euler's formula, for a weight homogeneous $F\in\mathcal A$, can be written as $\overrightarrow{\nabla}F\cdot\vec e_\varpi=\varpi(F)F.$
The isomorphism $\star^{\mu}$ allows us to transport the notion of weight homogeneity of skew-symmetric multi-derivations to K\"ahler $p$-differential forms.\\
Fixing weights $\varpi_1, \varpi_2, \varpi_3\in\mathbb N^\star$, it is clear that $\mathcal A={\bigoplus_{i\in\mathbb N}}\mathcal A_i$, where $\mathcal A_0=K$ and for $i\in\mathbb N^\star$, $\mathcal A_i$ is the $K$-vector space generated by all weight homogeneous polynomials of degree $i$. Denoting by $\Omega^k(\mathcal A)_i$ the $K$-vector space given by $\Omega^k(\mathcal A)_i=\{P\in\Omega^k(\mathcal A) : \varpi(P)=i\}\cup\{0\}$, we have the following isomorphisms : \\
\begin{align}
    \Omega^3(\mathcal A)_i\cong\mathcal X^0(\mathcal A)_i&\cong\mathcal A_i\nonumber\\
     \Omega^2(\mathcal A)_i\cong\mathcal X^1(\mathcal A)_i&\cong\mathcal A_{i+\varpi_1}\times\mathcal A_{i+\varpi_2}\times\mathcal A_{i+\varpi_3}\nonumber\\
     \Omega^1(\mathcal A)_i\cong\mathcal X^2(\mathcal A)_i&\cong\mathcal A_{i+\varpi_2+\varpi_3}\times\mathcal A_{i+\varpi_1+\varpi_3}\times\mathcal A_{i+\varpi_1+\varpi_2}\nonumber\\
     \Omega^0(\mathcal A)_i\cong\mathcal X^3(\mathcal A)_i&\cong\mathcal A_{i+\varpi_1+\varpi_2+\varpi_3}\nonumber
 \end{align}
\begin{rem}
As maps from $\mathcal X^k(\mathcal A)$ to $\mathcal X^{k-1}(\mathcal A),$ each arrow of the complex given by (\ref{qa}) is a weight homogeneous map of degree zero, while each arrow of complexes given by (\ref{hc1}) and (\ref{hc2}) is a weight homogeneous map of degree $\varpi(\lambda) +\varpi(P),$ if $\lambda$ and $P$ are weight homogeneous elements of $\mathcal A.$
\end{rem}
\subsection{The Koszul complex-complete intersection with an isolated singularity}
\begin{definition}
One say that a weight homogeneous element $U\in\mathcal A=K[x,y,z]$ has an isolated singularity in zero if
\begin{equation}\label{singularity}
    \mathcal A_{sing}(U):=K[x,y,z]/<\frac{\partial U}{\partial x},\frac{\partial U}{\partial y},\frac{\partial U}{\partial z}>
\end{equation}
has a finite dimension as a $K$-vector space.
\end{definition}
The dimension of $\mathcal A_{sing}(P)$ is called the Milnor number of the singular point. Usually, we will say a weight homogeneous polynomial has an isolated singularity to say that it has an isolated singularity in zero.\\
We shall now give a definition of dimension for rings. For this purpose, note that the length of the chain $P_r\supset P_{r-1}\supset\cdot\cdot\cdot\supset P_0$ involving $r+1$ distinct ideals of a given ring is taken to be $r.$
\begin{definition}
The Krull dimension of a ring $\mathcal R$ is the supremum of the lengths of chains of distinct prime ideals in $\mathcal R.$
\end{definition}
\begin{definition}
Let $\mathcal R$ be an associative and commutative graded $K$-algebra. A system of homogeneous elements $a_1,...,a_d$ in $\mathcal R$, where $d$ is the Krull dimension of $\mathcal R$, is called a homogeneous system of parameters of $\mathcal R$ (h.s.o.p.) if $\mathcal R/<a_1,...,a_d>$ is a finite dimensional $K$-vector space.
\end{definition}
For example, if we consider the $K$-algebra $\mathcal A=K[x,y,z]$, graded by the weight degree, we have a natural h.s.o.p. given by the system $x,y,z$.
 \begin{definition}
A sequence $a_1,...,a_n$ in a commutative associative algebra $\mathcal R$ is said to be an $\mathcal R$-regular sequence if $<a_1,...,a_n>\neq\mathcal R$ and $a_i$ is not a zero divisor of $\mathcal R/<a_1,...,a_{i-1}>$ for $i=1,2,...,n$.
\end{definition}
For any sequence $a_1,...,a_n$, we can define the following complex, called the associated Koszul complex, which is exact in the case of regular sequence (see Weibel ~\cite{wei}) :\\ \\
$\begin{array}{c}
0\longrightarrow\bigwedge^0(\mathcal R^n)\longrightarrow\cdot\cdot\cdot\longrightarrow\bigwedge^{n-2}(\mathcal R^n)\stackrel{\wedge \omega}{\longrightarrow}\bigwedge^{n-1}(\mathcal R^n)\stackrel{\wedge \omega}{\longrightarrow}\bigwedge^n(\mathcal R^n)\\
\end{array}$\\ \\
where $\omega=\displaystyle{\sum_{i=1}^n}a_ie_i$ and $(e_1, e_2,\cdots, e_n)$ is a basis of an $\mathcal R$-module free $\mathcal R^n.$\\
In our particular case, $\mathcal R=K[x_1,x_2,\cdots,x_n]$, using the identifications $\bigwedge^p(\mathcal R^n)\cong \Omega^p(\mathcal R)$, the Koszul complex associated to the sequence $\frac{\partial P}{\partial x_1},\frac{\partial P}{\partial x_2},\cdots,\frac{\partial P}{\partial x_n}$ ($P\in\mathcal R$) have the following form :

$\begin{array}{c}
0\longrightarrow\mathcal A\stackrel{\wedge dP}{\longrightarrow}\Omega^1(\mathcal A)\longrightarrow\cdot\cdot\cdot\longrightarrow\Omega^{n-2}(\mathcal A)\stackrel{\wedge dP}{\longrightarrow}\Omega^{n-1}(\mathcal A)\stackrel{\wedge dP}{\longrightarrow}\Omega^n(\mathcal A)\\
\end{array}$\\
Using the vector notation for $n=3$, we have the following complex : \\
$$0\longrightarrow\mathcal A\stackrel{\overrightarrow\nabla P}{\longrightarrow}\mathcal A^3\stackrel{\times\overrightarrow\nabla P}{\longrightarrow}\mathcal A^3\stackrel{\cdot\overrightarrow\nabla P}{\longrightarrow}\mathcal A.$$
\begin{theo}\textnormal{(Cohen-Macaulay)}. Let $\mathcal R$ be a noetherian graded $K$-algebra. If $\mathcal R$ has a h.s.o.p. which is a regular sequence, then any h.s.o.p. in $\mathcal R$ is a regular sequence.
\end{theo}
Thus, for each $P\in\mathcal A=K[x,y, z]$ which is a weight homogeneous polynomial with an isolated singularity, the sequence $\frac{\partial P}{\partial x},\frac{\partial P}{\partial y},\frac{\partial P}{\partial z}$ is regular and the associated Koszul complex is exact.
\section{Poisson homology of GJPS in dimension 3}
Let us consider the polynomial algebra $\mathcal A=K[x, y, z]$ where $K$ is a field of characteristic $0$ equipped with the GJPS $\pi$ given by two weight homogeneous polynomials $\lambda$ and the Casimir $P$ which define a regular sequence and $P$ has an isolated singularity.\\
First we compute the kernels $(ker\partial_i)_{i=1, 2, 3}.$
\begin{prop}\label{1pc}
$Ker\partial_1=\{\overrightarrow{G}=\overrightarrow{\nabla}F+ G\overrightarrow{\nabla}P, \ \ F, G\in\mathcal A\}$
\end{prop}
\begin{proof}
Let $\overrightarrow{G}$ be a weight homogeneous element of $\Omega^1(\mathcal A)\cong \mathcal A^3$ such that $\partial_1(\overrightarrow{G})=0.$ This is equivalent to say that $\overrightarrow{\nabla}P\cdot(\overrightarrow{\nabla}\times\overrightarrow{G})=0.$ According to the exactness of the Koszul complex associated with $P$, there exists $\overrightarrow{G}_1\in\mathcal A^3$ such that $\overrightarrow{\nabla}\times\overrightarrow{G}=\overrightarrow{G}_1\times\overrightarrow{\nabla}P.$\\
We get $0=Div(\overrightarrow{\nabla}\times\overrightarrow{G})=\overrightarrow{\nabla}P\cdot(\overrightarrow{\nabla}\times\overrightarrow{G}),$ and therefore $\overrightarrow{G}_1$ is also a Poisson cycle.

Proceeding in this way yields the existence of a sequence of Poisson cycle $(\overrightarrow{G}_n)_{n\in\mathbb N},$ with $\overrightarrow{G}=\overrightarrow{G}_0,$ such that
$\overrightarrow{\nabla}\times\overrightarrow{G}_k=\overrightarrow{G}_{k+1}\times\overrightarrow{\nabla}P$ and $deg \overrightarrow{G}_{k+1}< deg\overrightarrow{G}_{k}.$ Then there exists $m\in \mathbb N$ such that $\overrightarrow{\nabla}\times\overrightarrow{G}_m=0.$ Using the Poincaré lemma, there exists $F_m\in\mathcal A$ such that $\overrightarrow{G}_m=\overrightarrow{\nabla}F_m.$

Since $\overrightarrow{\nabla}\times\overrightarrow{G}_{m-1}=\overrightarrow{G}_{k+1}\times\overrightarrow{\nabla}P=\overrightarrow{\nabla}\times (F_m\overrightarrow{\nabla}P),$ by the Poincaré lemma, we have $\overrightarrow{G}_{m-1}=F_m\overrightarrow{\nabla}P+ \overrightarrow{\nabla}F_{m-1}.$
Finally, we get $\overrightarrow{G}=G\overrightarrow{\nabla}P+ \overrightarrow{\nabla}F,$ $F, G\in\mathcal A.$
\end{proof}
\begin{lem}\label{cas}
Let $F\in\mathcal A.$ $\overrightarrow{\nabla}F\times \overrightarrow{\nabla}P=0$ iff $F\in K[P].$
\end{lem}
\begin{proof}
Let $F$ be a weight homogeneous polynomial such that $\overrightarrow{\nabla}F\times \overrightarrow{\nabla}P=0.$ We will proceed by induction on the degree of $F.$ If $\varpi(F)\leqq0,$ the result is obvious. Let us suppose now that the result is true for all polynomial of degree less than $\varpi(F).$ Using the exactness of the Koszul complex associated with $P,$ there exists a polynomial $F_1$ such that $\overrightarrow{\nabla}F=F_1\overrightarrow{\nabla}P.$ Therefore, we get $\overrightarrow{\nabla}F_1\times\overrightarrow{\nabla}P=0$ and $\varpi(F_1)<\varpi(F).$ Then $F_1\in K[P]$ and hence $F\in K[P].$
\end{proof}
\begin{lem}\label{reg} Let $F$ be a weight homogenous element of $\mathcal A.$
If $\lambda F=cP^r,$ $c\in K$ and $r\in\mathbb N,$ then $c=0.$
\end{lem}
\begin{proof}
We will proceed by induction on $r.$ If $r=0,$ then $F\lambda$ is a constant. This is possible only if $F=0$ since $\lambda\neq0.$ Then $c=0.$ Now, let us suppose that the hypothesis is true for some $r\in\mathbb N.$ Let $F$ be a weight homogeneous polynomial such that $\lambda F=cP^{r+1}.$ Since the sequence $(\lambda, P)$ is regular, $\lambda$ is not a divisor of zero in $\mathcal A/(P).$ Then there exists $F_1\in\mathcal A$ such that $F=F_1P.$ Therefore $\lambda F_1=cP^{r}.$ Using the induction supposition, we conclude that $c=0.$
\end{proof}
\begin{prop}\label{2pc}
$Ker\partial_2=\{\overrightarrow{G}=\overrightarrow{\nabla}H\times\overrightarrow{\nabla}P, \ \ H\in\mathcal A\}$
\end{prop}
\begin{proof}
Let $\overrightarrow{G}$ be a weight homogeneous element of $\Omega^2(\mathcal A)\cong \mathcal A^3$ such that $\partial_2(\overrightarrow{G})=0.$ Then $-\overrightarrow{\nabla}(\lambda\overrightarrow{G}\cdot\overrightarrow{\nabla}P)+ \lambda Div(\overrightarrow{F})\overrightarrow{\nabla}P=0$ and therefore $\overrightarrow{\nabla}(\lambda\overrightarrow{G}\cdot\overrightarrow{\nabla}P)\times \overrightarrow{\nabla}P=0.$ Using the lemma \ref{cas}, $\lambda\overrightarrow{G}\cdot\overrightarrow{\nabla}P$ is a Casimir. Let us suppose that $\lambda\overrightarrow{G}\cdot\overrightarrow{\nabla}P=cP^r,$ $c\in K$ and $r$ is an integer. From lemma \ref{reg}, $c=0$ and then $\overrightarrow{G}\cdot\overrightarrow{\nabla}P=0.$ Therefore $Div(\overrightarrow{G})=0.$ From the Poincaré lemma, $\overrightarrow{G}=\overrightarrow{\nabla}\times\overrightarrow{H}$ and hence $(\overrightarrow{\nabla}\times\overrightarrow{H})\cdot\overrightarrow{\nabla}P=0.$ Using proposition \ref{1pc}, we conclude that there exists $F, G\in\mathcal A$ such that $\overrightarrow{H}=H\overrightarrow{\nabla}P+ \overrightarrow{\nabla}F.$ Then $\overrightarrow{G}=\overrightarrow{\nabla}H\times\overrightarrow{\nabla}P.$
\end{proof}
\begin{prop}
$Ker\partial_3=\{0\}$
\end{prop}
\begin{proof}
Let $F$ be a weight homogeneous element of $\Omega^3(\mathcal A)\cong \mathcal A$ such that $\partial_3(F)=0.$ Then $\overrightarrow{\nabla}(\lambda F)\times\overrightarrow{\nabla}P=0.$ Then $\lambda F$ is a Casimir. Let us suppose that $\lambda F=cP^r,$ $c\in K$ and $r\in \mathbb N.$ According to lemma (\ref{reg}), $c=0$ and therefore $F=0.$
\end{proof}
\begin{rem}
One can surprisingly remark that the Poisson cycles do not depend on the polynomial $\lambda.$
\end{rem}
\begin{cor}
$PH_3(\mathcal A, \pi)=0.$
\end{cor}
\begin{prop}
The following complexes are exact and the arrows are maps of degree zero:
\begin{equation}\label{z}
    0\longrightarrow K[P]\stackrel{\alpha}{\longrightarrow}\mathcal A(-\varpi(P))\oplus\mathcal A\stackrel{\beta}{\longrightarrow}ker\partial_1\longrightarrow 0
\end{equation}\\
$\alpha(u(P))=(-\frac{\partial u}{\partial P}, u)$; \\ $\beta(\alpha_1, \alpha_2)=\alpha_1\overrightarrow\nabla P+\overrightarrow\nabla\alpha_2.$ \\
\begin{equation}\label{y}
   0\longrightarrow K[P](-\varpi(P))\stackrel{\gamma}{\longrightarrow}\mathcal A(-\varpi(P))\stackrel{\epsilon}{\longrightarrow}ker\partial_2\longrightarrow 0
\end{equation}
$\gamma(u(P))=u(P)$; \\ $\epsilon(H)=\nabla H\times\overrightarrow\nabla P.$\\
\end{prop}
\begin{proof}
Let us give the proof for the first sequence:\\
If $\alpha_1\overrightarrow\nabla P+\overrightarrow\nabla\alpha_2=0$, then $(\alpha_1\overrightarrow\nabla P+\overrightarrow\nabla\alpha_2)\times \overrightarrow\nabla P=0.$\\
Therefore, $\overrightarrow\nabla\alpha_2\times \overrightarrow\nabla P=0$ and $\alpha_2\in K[P].$ Then $\alpha_1=-\frac{\partial\alpha_2}{\partial P}.$\\
We can conclude that the first complex is exact. The proof of the exactness of the $2nd$ complex is similar to the first one.
\end{proof}
We also have the following trivial exact complex sequence :\\
\begin{equation}\label{w}
    0\longrightarrow ker\partial_{i+1}\longrightarrow\Omega^{i+1}(\mathcal A)\longrightarrow ker\partial_i\longrightarrow PH_i(\mathcal A, \pi)\longrightarrow 0
\end{equation}
where $i=0, 1, 2,$ and the arrows are maps of degree zero.\\
\begin{rem}
By using the exactness of the complexes (\ref{z}), (\ref{y}), (\ref{w}), we obtain the Poincaré's series of Poisson homology groups. We can obviously notice that these series do not depend on the polynomials $\lambda$ and $P$, but on weights and the degree of $P,$ and it is surprising to observe that these series do not depend on the weight of $\lambda.$
\end{rem}
We are going to explicitly calculate these Poincaré's series when the Casimir $P$ is quadratic.\\
\begin{theo}
If $\varpi_1=\varpi_2=\varpi_3=1$ and $\varpi(\lambda)=1,$ $\varpi(P)=2$, then as K-vector spaces, the Poisson homological groups $H_i(\Omega^\bullet(\mathcal A), \pi), i=1,2,3$ have the following Poincaré series:\\
$$\begin{array}{lcl}
P(PH_0(\mathcal A, \pi),t)&=&\frac{-t^2+2t+1}{(1-t^2)(1-t)} ;\\
&&\\
P(PH_1(\mathcal A, \pi),t)&=&\frac{t(2t^2+t+1}{(1-t^2)(1-t)};\\
&&\\
P(PH_2(\mathcal A, \pi),t)&=&\frac{2t^3}{(1-t^2)(1-t)}; \\
&&\\
P(PH_3(\mathcal A, \pi),t)&=&0. \\
\end{array}$$
\end{theo}
An example of such quadratic polynomial Poisson has been considered recently by Gurevich and Saponov in their paper \cite{gur}:
 \begin{equation}\label{exgur}
    \lambda=z,\ \ \  \ P= xy+\frac{1}{2}z^2.\\
 \end{equation}
   Another example is a particular case of a family of GJPS mentioned in the PhD thesis of Anne Pichereau in 2006 : $\lambda=z$ and $P=\frac{1}{2}(x^2+y^2+z^2).$ The general case considered by Pichereau is to set:
 \begin{equation}\label{expich}
\lambda=z, \ \ \ \  P=\frac{1}{n+1}(x^{n+1}+y^{n+1}+z^{n+1}).
\end{equation}

\begin{rem}
The first observation is that the third Poisson homological group of GJPS is always trivial. That was not the case for of JPS \cite{pic}. This confirms also the fact that the modular class is non trivial since in the trivial case, one would get, by Poincaré duality that $PH_3$ is isomorphic to $PH^0\cong K[P].$ The second remark explains  the fact that, as module over the Casimir algebra, $K[P]$, for $i=0, 1, 2,$ the Poisson homological group $PH_i(\mathcal A, \pi)$ is not of finite type as it is the case for quadritic JPS \cite{pic}.
\end{rem}

\section{Poisson cohomology of GJPS in dimension 3}
This section is devoted to the computation of the Poisson cohomology associated with a GJPS on $\mathcal A=K[x, y, z]$ given by  the polynomials $\lambda$ and $P,$ as Casimir, with additional conditions. We suppose that:
\begin{enumerate}
  \item $P$ is weight homogeneous according with some weights $\varpi_1,$ $\varpi_2,$ $\varpi_3,$ for the variables $x, y, z$ respectively;
  \item $P=\widetilde{P}+\frac{c}{r+2}z^{r+2},$ where $\widetilde{P}\in K[x,y],$ $c\in K,$ and $r\in\mathbb N;$
  \item $P$ has an isolated singularity;
  \item $\lambda=z.$
\end{enumerate}
It is interesting to note these assumptions satisfy the conditions considered in the section concerning the Poisson homology.\\
A direct consequence of these assumptions is that $\widetilde{P}$ is a weight homogeneous in $K[x, z]$ for some weights $\varpi'_1$ and $\varpi'_2$ associated to variables $x$ and $y$ respectively, may be different from $\varpi_1, \varpi_2,$ and has isolated singularity.\\
Examples of such GJPS are those considered by Gurevich-Saponov and Pichereau (\ref{exgur}), (\ref{expich}). These examples are homogeneous. Let us give a nonhomogeneous case:
\begin{equation}\label{nh}
\lambda=z, \ \ \ \  P=(x^2+y^2+z^3).
\end{equation}
In this last case, one observes that $\varpi'_i\neq\varpi_i,$ $i=1, 2.$
\subsection{Vectorial notations in dimension two}
Set $\mathcal B=K[x, y].$ We have the following isomorphisms of $\mathcal B$-modules:
$$\begin{array}{ccc}\Omega^1(\mathcal B)&\stackrel{\cong}{\longrightarrow} &\mathcal B^2\\
F_1dx+F_2dy&\longmapsto &(F_1,F_2)\\
\end{array}$$
$$\begin{array}{ccc}\Omega^2(\mathcal B)&\stackrel{\cong}{\longrightarrow} &\mathcal B\\
Fdx\wedge dy&\longmapsto &F\\
\end{array}$$
$$\begin{array}{ccc}\mathcal X^1(\mathcal B)&\stackrel{\cong}{\longrightarrow} &\mathcal B^2\\
F_1\frac{\partial}{\partial x}+F_2\frac{\partial}{\partial y}&\longmapsto &(F_1,F_2)\\
\end{array}$$
$$\begin{array}{ccc}\mathcal X^2(\mathcal B)&\stackrel{\cong}{\longrightarrow} &\mathcal B\\
F\frac{\partial}{\partial x}\wedge \frac{\partial}{\partial y}&\longmapsto &F\\
\end{array}$$
Then using these vectorial notations, the de Rham complex for the algebra $\mathcal B$ can be written as follows:
\begin{equation}\label{drc1}
  K{\hookrightarrow}\mathcal B\stackrel{d_0}{\longrightarrow}\mathcal B^2\stackrel{d_1}{\longrightarrow}\mathcal B{\longrightarrow} 0,
\end{equation}
where $d_0(F)=\left(\frac{\partial F}{\partial x}, \frac{\partial F}{\partial y}\right)^t;$  $d_1(\overrightarrow{F})=\frac{\partial F_2}{\partial x}-\frac{\partial F_1}{\partial y}.$\\
And from the $\star$-isomorphism between the Kähler differential forms and the skew-symmetric multi-derivation, the de Rham complex could take this second following form:
\begin{equation}\label{drc2}
\begin{array}{cccccccccc}
  K&{\hookrightarrow}&\mathcal B & \stackrel{d}{\longrightarrow} & \mathcal B^2 & \stackrel{d}{\longrightarrow} & \mathcal B&{\longrightarrow} &0\\
   & &F & \longmapsto & \overrightarrow{\Box}F &  &  &  & \\
&  &  &  &\overrightarrow{K} & \longmapsto & \mbox{Div}(\overrightarrow{K}) & &
\end{array}
\end{equation}
where $\overrightarrow{\Box}(F)=\left(\frac{\partial F}{\partial y}, -\frac{\partial F}{\partial x}\right)^t;$ $Div(\overrightarrow{K})=\frac{\partial K_1}{\partial x}+\frac{\partial K_2}{\partial y},$ $\overrightarrow{K}=(K_1, K_2)^t\in\mathcal B^2.$\\
It is well known that this de Rham complex is exact.\\
Now let $\varphi\in\mathcal B.$ The associated Koszul complex takes the following form:
\begin{equation}\label{kc1}
\begin{array}{cccccccc}
  0&{\hookrightarrow}&\mathcal B & \stackrel{\overrightarrow{\nabla}\varphi}{\longrightarrow} & \mathcal B^2 & \stackrel{\cdot\overrightarrow{\Box}\varphi}{\longrightarrow} & \mathcal B
  \end{array}
\end{equation}
As for the de Rham complex, using the $\star$-isomorphism between the Kähler differential forms and the skew-symmetric multi-derivation, the Koszul complex associated to $\varphi\in\mathcal B$ can be rewritten in the following second form:
\begin{equation}\label{kc2}
\begin{array}{cccccccc}
  0&{\hookrightarrow}&\mathcal B & \stackrel{\overrightarrow{\Box}\varphi}{\longrightarrow} & \mathcal B^2 & \stackrel{\cdot(-\overrightarrow{\nabla}\varphi)}{\longrightarrow} & \mathcal B
  \end{array}
\end{equation}
Fixing weights $\varpi_1', \varpi_2'\in\mathbb N^\star$ for the variables $x, y$, it is clear that $\mathcal B={\bigoplus_{i\in\mathbb N}}\mathcal B_i$, where $\mathcal B_0=K$ and for $i\in\mathbb N^\star$, $\mathcal B_i$ is the $K$-vector space generated by all weight homogeneous polynomials of degree $i$. Denoting by $\Omega^k(\mathcal B)_i$ the $K$-vector space given by $\Omega^k(\mathcal B)_i=\{P\in\Omega^k(\mathcal B) : \varpi(P)=i\}\cup\{0\}$, we have the following isomorphisms : \\
\begin{align}
    \Omega^2(\mathcal B)_i\cong\mathcal X^0(\mathcal B)_i&\cong\mathcal B_i\nonumber\\
     \Omega^1(\mathcal B)_i\cong\mathcal X^1(\mathcal B)_i&\cong\mathcal B_{i+\varpi_1}\times\mathcal B_{i+\varpi_2}\nonumber\\
     \Omega^0(\mathcal B)_i\cong\mathcal X^2(\mathcal B)_i&\cong\mathcal B_{i+\varpi_1+\varpi_2}\nonumber
 \end{align}
\begin{rem}
As maps from $\mathcal X^k(\mathcal B)$ to $\mathcal X^{k-1}(\mathcal B)$ each arrow of the complex given by (\ref{drc2}) is a weight homogeneous map of degree zero, while each arrow of the complex given by (\ref{kc2}) is a weight homogeneous map of degree $\varpi(\varphi),$ if $\varphi$ is a weight homogeneous element of $\mathcal B.$
\end{rem}
\subsection{Computation of Poisson cohomology of GJPS in dimension 3}
 We want to describe the Poisson cohomology associated with a GJPS on $\mathcal A=K[x, y, z]$ given by  the polynomials $\lambda$ and $P$ as Casimir. We suppose that:
\begin{enumerate}\label{cond}
  \item $P$ is weight homogeneous according for some weights $\varpi_1,$ $\varpi_2,$ $\varpi_3,$ for the variables $x, y, z$ respectively;
  \item $P=\widetilde{P}+\frac{1}{r+2}z^{r+2},$ where $\widetilde{P}\in K[x,y]$ and $r\in\mathbb N;$
  \item $P$ has an isolated singularity;
  \item $\lambda=z.$
\end{enumerate}
\begin{theo}
$PH^0(\mathcal A, \pi)=K[P].$
\end{theo}
\begin{proof}
It is a direct consequence of lemma \ref{cas}.
\end{proof}
Since $\widetilde{P}$ is weight homogeneous in $\mathcal B=K[x, y],$ for some weights $\varpi_1'$ and $\varpi_2'$ of variables $x$ and $y$ respectively and has an isolated singularity, $\mathcal B_{sing}(\widetilde{P})=\frac{\mathcal B}{\left\{\frac{\partial \widetilde{P} }{\partial x}, \frac{\partial \widetilde{P} }{\partial y}\right\}}$ has a finite dimension as a $K$-vector space.  We denote by $\mu(\widetilde{P})$ the associated Milnor number and  by $\vartheta_0=1,\cdots, \vartheta_{\mu(\widetilde{P})-1},$ a basis of the vector space $\mathcal B_{sing}(\widetilde{P}).$ We will denote by $\varpi'(F)$ the weight homogeneous degree of a weight homogeneous element $F$ of $\mathcal B.$\\
We have the same result for $P\in\mathcal A=K[x, y, z].$ We will denote by $\mu_0=1,\cdots, \mu_{\mu(P)-1}$ a basis of the $K$-vector space $\mathcal A_{sing}(P).$
\begin{lem}\label{lema} We have the following isomorphism of $K$-vector spaces:
$$\frac{\mathcal B}{\left\{\overrightarrow{\Box} Q\cdot\overrightarrow{\nabla}\widetilde{P},\  Q\in\mathcal B\right\}}=\displaystyle{\bigoplus_{i\in \mathbb N}}\displaystyle{\bigoplus_{j=0}^{\mu(\widetilde{P})-1}}K\widetilde{P}^i\vartheta_j.$$
\end{lem}
\begin{proof}
Let $F\in\mathcal B$ be a weight homogeneous polynomial. Our purpose is to prove that there exist $Q\in\mathcal B$ and constants $\alpha_{i,j}\in K$ such that:
\begin{equation}\label{lema1}
F=\overrightarrow{\Box} Q\cdot\overrightarrow{\nabla}\widetilde{P}+\displaystyle{\sum_{i\in \mathbb N}}\displaystyle{\sum_{j=0}^{\mu(\widetilde{P})-1}}\alpha_{i,j}\widetilde{P}^i\vartheta_j.
\end{equation}

We have:
$$F=\overrightarrow{G}\cdot\overrightarrow{\nabla}\widetilde{P}+\displaystyle{\sum_{j=0}^{\mu(\widetilde{P})-1}}\alpha_j\vartheta_j,$$
where $\overrightarrow{G}\in\mathcal X^1(\mathcal B)_{\varpi'(F)-\varpi'(\widetilde{P})},$ and $\alpha_0,\cdots, \alpha_{\mu(\widetilde{P})-1}\in K.$ Let $\widetilde{\varpi}=max(\varpi_1', \varpi_2').$\\
If $\varpi'(F)<\varpi'(\widetilde{P})-\widetilde{\varpi},$ then $\overrightarrow{G}=(a,b)^t\in K^2.$ Therefore $\overrightarrow{G}=\overrightarrow{\Box}Q,$ $Q=-bx+ay.$ And we get the expected result.\\
Now let us suppose that $\varpi'(F)\geq\varpi'(\widetilde{P})-\widetilde{\varpi}$ and that for any weight homogeneous polynomial $L\in\mathcal B$ of degree less than $\varpi'(F)-1,$ an equality of the form (\ref{lema1}) holds.
Since
$$Div\left(\overrightarrow{G}-\frac{1}{\varpi'(F)-\varpi'(\widetilde{P})+|\varpi'|}Div(\overrightarrow{G})\vec e_{\varpi'}\right)=0,$$
where $|\varpi'|=\varpi_1'+\varpi_2';$  $\vec e_{\varpi'}=(\varpi_1'x_1, \varpi_2'x_2)^t,$ and according to exactness of the de Rham complex \ref{drc2}, there exists a weight homogeneous polynomial $U\in\mathcal B$ such that:
$$\overrightarrow{G}=\overrightarrow{\Box}U+\frac{1}{\varpi'(F)-\varpi'(\widetilde{P})+|\varpi'|}Div(\overrightarrow{G})\vec e_{\varpi'}.$$
Then
$$\overrightarrow{G}\cdot\overrightarrow{\nabla}\widetilde{P}=\overrightarrow{\Box}U\cdot\overrightarrow{\nabla}\widetilde{P}+\frac{\varpi'(\widetilde{P})}
{\varpi'(F)-\varpi'(\widetilde{P})+|\varpi'|}Div(\overrightarrow{G})\widetilde{P}.$$
Since $\varpi'(Div(\overrightarrow{G}))=\varpi'(F)-\varpi'(\widetilde{P})$ and according to the recursion hypothesis, there exist $V\in\mathcal B$ and constants $\beta_{i,j}\in K$ such that:
$$Div(\overrightarrow{G})=\overrightarrow{\Box} V\cdot\overrightarrow{\nabla}\widetilde{P}+\displaystyle{\sum_{i\in \mathbb N}}\displaystyle{\sum_{j=0}^{\mu(\widetilde{P})-1}}\beta_{i,j}\widetilde{P}^i\vartheta_j.$$
Therefore, we obtain the expected result.
\end{proof}
\begin{prop}\label{p1}
 We have the following isomorphism of $K$-vector spaces:
$$\frac{\mathcal A}{\left\{(\overrightarrow{\nabla}\lambda\times\overrightarrow{\nabla}P)\cdot\overrightarrow{\nabla}Q,\  Q\in\mathcal A\right\}}=\displaystyle{\bigoplus_{i, j\in \mathbb N}}\displaystyle{\bigoplus_{k=0}^{\mu(\widetilde{P})-1}}K\widetilde{P}^i\lambda^j\vartheta_k.$$
\end{prop}
\begin{proof}
Let $F\in\mathcal A=K[x, y, z].$ We have:
$$F=\displaystyle{\sum_{j\in\mathbb N}}F_j\lambda^j,$$
where $F_i\in\mathcal B.$ According to the previous lemma, there exist polynomials $Q_j\in\mathcal B$ and some constants $\alpha_{i,j,k}\in K$ such that
$$F_j=\overrightarrow{\Box} Q_j\cdot\overrightarrow{\nabla}\widetilde{P}+\displaystyle{\sum_{i\in \mathbb N}}\displaystyle{\sum_{k=0}^{\mu(\widetilde{P})-1}}\alpha_{i,j,k}\widetilde{P}^i\vartheta_k.$$
Since $$\overrightarrow{\Box} Q_j\cdot\overrightarrow{\nabla}\widetilde{P}=(\overrightarrow{\nabla}\lambda\times\overrightarrow{\nabla}P)\cdot\overrightarrow{\nabla}Q_i$$
 and
$$\lambda^j(\overrightarrow{\nabla}\lambda\times\overrightarrow{\nabla}P)\cdot\overrightarrow{\nabla}Q_i=
(\overrightarrow{\nabla}\lambda\times\overrightarrow{\nabla}P)\cdot\overrightarrow{\nabla}(\lambda^jQ_i),$$
we get the expected result.
\end{proof}
\begin{prop}\label{p2}\cite{pic} Let $F\in\mathcal A=K[x, y, z].$ There exists $\overrightarrow{G}\in\mathcal A^3$ such that $$F\in\overrightarrow{\nabla}P\cdot\left(\overrightarrow{\nabla}\times\overrightarrow{G}\right)+\displaystyle{\bigoplus_{s=0}^{\mu(P)-1}}K[P]\mu_s.$$
\end{prop}
\begin{lem}\label{lem1}
$\varpi_1+\varpi_2-\varpi_3$ and $\frac{\varpi(P)}{\varpi'(\widetilde{P})}\left(|\varpi'|-\frac{\varpi'(\widetilde{P})}{r+2}\right)$ have the same sign.
\end{lem}
\begin{proof}
Set $\alpha=pgcd(\varpi_1, \varpi_2).$ Then $\varpi_1'=\frac{\varpi_1}{\alpha}$ and $\varpi_2'=\frac{\varpi_2}{\alpha}.$
Note that $\varpi(P)=\varpi_3(r+2).$ From
$$\frac{\varpi_1}{\alpha}x\frac{\partial\widetilde{P}}{\partial x}+\frac{\varpi_2}{\alpha}y\frac{\partial\widetilde{P}}{\partial y}=\varpi'(\widetilde{P})\widetilde{P},$$
we obtain
$$\varpi_1x\frac{\partial{P}}{\partial x}+\varpi_2y\frac{\partial{P}}{\partial y}+\frac{\alpha\varpi'(\widetilde{P})}{r+2}z\frac{\partial{P}}{\partial z}=\alpha\varpi'(\widetilde{P}){P}.$$
Then $\varpi(P)=\alpha\varpi(\widetilde{P}),$ $\varpi_3=\frac{\alpha\varpi'(\widetilde{P})}{r+2}$ and we deduce that
$$\varpi'_1+\varpi'_2-\frac{\varpi'(\widetilde{P})}{r+2}=\frac{1}{\alpha}\left(\varpi_1+\varpi_2-\varpi_3\right).$$
\end{proof}
\begin{theo}\label{theo3}
Suppose that $\varpi_1+\varpi_2-\varpi_3\geq0.$ Then the third Poisson cohomological group associated to the GJPS given by  the polynomials $\lambda$ and $P$ as Casimir, which satisfy the conditions (\ref{cond}), is given by:
$$PH^3(\mathcal A, \pi)=K[P]\otimes\mathcal A_{sing}(P'),$$
where $P'=\widetilde{P}+\frac{1}{r+3}z^{r+3}.$
\end{theo}
\begin{proof}
Let $F$ be a weight homogeneous element of $\mathcal A.$ From the previous proposition \ref{p1}, there exist $Q\in\mathcal A$ and constants $\alpha_{i,j,k}\in K$ such that:
\begin{equation}
F=\overrightarrow{\nabla}Q\cdot\left(\overrightarrow{\nabla}\lambda\times\overrightarrow{\nabla}P\right)+\displaystyle{\sum_{i, j\in \mathbb N}}\displaystyle{\sum_{k=0}^{\mu(\widetilde{P})-1}}\alpha_{i,j,k}\widetilde{P}^i\lambda^j\vartheta_k.
\end{equation}
But one can observe that $\overrightarrow{\nabla}Q\cdot\left(\overrightarrow{\nabla}\lambda\times\overrightarrow{\nabla}P\right)=\delta^2(-\overrightarrow{\nabla}Q).$ Then
$$\mathcal A=Im\delta^2+\displaystyle{\sum_{i, j\in \mathbb N}}\displaystyle{\sum_{k=0}^{\mu(\widetilde{P})-1}}K\widetilde{P}^i\lambda^j\vartheta_k.$$
Now fix $j\geq 2.$ Using the Proposition \ref{p2}, there exists $\overrightarrow{G}\in\mathcal A^3$ such that:
$$\widetilde{P}\lambda^{j-2}\vartheta_k\in\overrightarrow{\nabla}P\cdot\left(\overrightarrow{\nabla}\times\overrightarrow{G}\right)+\displaystyle{\sum_{l\in \mathbb N}}\displaystyle{\sum_{s=0}^{\mu({P})-1}}K{P}^l\mu_s.$$
Since $\delta^2(-\lambda\overrightarrow{G})=\lambda^2\overrightarrow{\nabla}P\cdot\left(\overrightarrow{\nabla}\times\overrightarrow{G}\right),$ we get:
$$\widetilde{P}\lambda^{j}\vartheta_k\in Im\delta^2+ \displaystyle{\sum_{l\in \mathbb N}}\displaystyle{\sum_{s=0}^{\mu({P})-1}}K\lambda^2{P}^l\mu_s.$$
Therefore
$$\mathcal A=Im\delta^2+\displaystyle{\sum_{l\in \mathbb N}}\displaystyle{\sum_{s=0}^{\mu({P})-1}}K\lambda^2{P}^l\mu_s+ \displaystyle{\sum_{i\in \mathbb N}}\displaystyle{\sum_{k=0}^{\mu(\widetilde{P})-1}}K\widetilde{P}^i\vartheta_k+ \displaystyle{\sum_{j\in \mathbb N}}\displaystyle{\sum_{k=0}^{\mu(\widetilde{P})-1}}K\widetilde{P}^j\lambda\vartheta_k.$$
On the other hand, since $\widetilde{P}=P-\frac{1}{r+2}\lambda^{r+2},$ there exists $Q_i\in\mathcal A$ such that $\widetilde{P}^i=P^i+\lambda^{r+2}Q_i.$ Then $\widetilde{P}^i\vartheta_k=P^i\vartheta_k+\lambda^{r+2}Q_i\vartheta_k.$\\
But according to Proposition \ref{p2}, there exists $\overrightarrow{G}_i\in\mathcal A^3$ such that
$$\lambda^rQ_i\vartheta_k\in\overrightarrow{\nabla}P\cdot\left(\overrightarrow{\nabla}\times\overrightarrow{G}\right)+\displaystyle{\sum_{l\in \mathbb N}}\displaystyle{\sum_{s=0}^{\mu({P})-1}}K{P}^l\mu_s,$$
and we have
$$\lambda^{r+2}Q_i\vartheta_k\in Im\delta^2+\displaystyle{\sum_{l\in \mathbb N}}\displaystyle{\sum_{s=0}^{\mu({P})-1}}K\lambda^2{P}^l\mu_s.$$
Therefore
$$\widetilde{P}^i\vartheta_k\in Im\delta^2+P^i\vartheta_k+\displaystyle{\sum_{l\in \mathbb N}}\displaystyle{\sum_{s=0}^{\mu({P})-1}}K\lambda^2{P}^l\mu_s.$$
In the same way, we get
$$\lambda\widetilde{P}^i\vartheta_k\in Im\delta^2+\lambda P^i\vartheta_k+\displaystyle{\sum_{l\in \mathbb N}}\displaystyle{\sum_{s=0}^{\mu({P})-1}}K\lambda^2{P}^l\mu_s.$$
Then
$$\mathcal A=Im\delta^2+\displaystyle{\sum_{s=0}^{\mu({P})-1}}K[P]\lambda^2\mu_s+
\displaystyle{\sum_{k=0}^{\mu(\widetilde{P})-1}}K[P]\lambda\vartheta_k+\displaystyle{\sum_{k=0}^{\mu(\widetilde{P})-1}}K[P]\vartheta_k.$$
From the canonical isomorphism of $K$-vector spaces
$$\mathcal A_{sing}(P)=\frac{K[x, y, z]}{\left(\frac{\partial\widetilde{P}}{\partial x}, \frac{\partial\widetilde{P}}{\partial y}, z^{r+1}\right)}\cong\left(K[x, y, z]/(z^{r+1})\right)/\left(\frac{\partial\widetilde{P}}{\partial x}, \frac{\partial\widetilde{P}}{\partial y}\right),$$
and according to the fact that $\widetilde{P}\in K[x, y],$ we suppose that $\{\mu_s\}=\{\vartheta_k, \lambda\vartheta_k,\cdots, \lambda^r\vartheta_k\}.$
Therefore
$$\mathcal A=Im\delta^2+
\displaystyle{\sum_{i=0}^{r+2}}\displaystyle{\sum_{k=0}^{\mu(\widetilde{P})-1}}K[P]\lambda^i\vartheta_k.$$
Now set $\overrightarrow{V}_k=(\varpi'_2y\vartheta_k, -\varpi'_1x\vartheta_k, 0)^t.$ We have:
$$\begin{array}{rcl}
\delta^2(\overrightarrow{V}_k)&=&-\lambda\overrightarrow{\nabla}P\cdot\left(\overrightarrow{\nabla}\times\overrightarrow{V}_k\right)-\overrightarrow{V}_k\cdot
\left(\overrightarrow{\nabla}\lambda\times\overrightarrow{\nabla}P\right)\\
&=&(|\varpi'|+\varpi'(\vartheta_k))\lambda^{r+2}\vartheta_k+\varpi'(\widetilde{P})\widetilde{P}\vartheta_k\\
&=&(\varpi'(\vartheta_k)+|\varpi'|-\frac{\varpi'(\widetilde{P})}{r+2})\lambda^{r+2}\vartheta_k+\varpi'(\widetilde{P})P\vartheta_k.
\end{array}$$
Then
\begin{equation}\label{dcomp1}
\mathcal A=Im\delta^2+
\displaystyle{\sum_{i=0}^{r+1}}\displaystyle{\sum_{k=0}^{\mu(\widetilde{P})-1}}K[P]\lambda^i\vartheta_k.
\end{equation}
Let us now prove that the sum (\ref{dcomp1}) is a direct one. First of all note that
\begin{equation}\label{eq0}
\frac{\mathcal A}{\left(\frac{\partial\widetilde{P}}{\partial x}, \frac{\partial\widetilde{P}}{\partial y}, \lambda^{r+2}\right)}=
\displaystyle{\bigoplus_{i=0}^{r+1}\bigoplus_{k=0}^{\mu(\widetilde{P})-1}}K\lambda^i\vartheta_k.
\end{equation}

Let us suppose the contrary. Consider $j_0$ the smallest integer such that there exists an equation of the form:
\begin{equation}\label{eq1}
\delta^2(\overrightarrow{F})=\displaystyle{\sum_{i=0}^{r+1}}\displaystyle{\sum_{j\geq j_0}}\displaystyle{\sum_{k=0}^{\mu(\widetilde{P})-1}}\alpha_{i,j,k}\lambda^iP^j\vartheta_k,
\end{equation}
where $\overrightarrow{F}\in\mathcal A^3,$ $\alpha_{i, j, k}\in K$ and $\alpha_{i_0, j_0, k_0}\neq 0,$ for at least one $i_0\in \mathbb N$ and for some $k_0\in\{0,\cdots, \mu(\widetilde{P})-1\}.$\\
Suppose that $j_0=0.$ Then\\
$$\begin{array}{rcl}
\displaystyle{\sum_{i=0}^{r+1}}\displaystyle{\sum_{k=0}^{\mu(\widetilde{P})-1}}\alpha_{i,0,k}\lambda^i\vartheta_k&=&-\displaystyle{\sum_{i=0}^{r+1}}\displaystyle{\sum_{j\in\mathbb N^{\star}}}\displaystyle{\sum_{k=0}^{\mu(\widetilde{P})-1}}\alpha_{i,j,k}\lambda^iP^j\vartheta_k\\
&&-\lambda\overrightarrow{\nabla}P\cdot\left(\overrightarrow{\nabla}\times\overrightarrow{F}\right)-\overrightarrow{F}\cdot
\left(\overrightarrow{\nabla}\lambda\times\overrightarrow{\nabla}P\right)\\
&\in&\left(\frac{\partial\widetilde{P}}{\partial x}, \frac{\partial\widetilde{P}}{\partial y}, \lambda^{r+2}\right).
\end{array}$$
Therefore, according to (\ref{eq0}), $\alpha_{i,0, k}=0,$ for all $i\in\{0,\cdots, r+1\}$ and all $k\in\{0,\cdots, \mu(\widetilde{P})-1\}.$ That is a contradiction with the hypothesis. Then $j_0\geq 1,$ and using the Euler formula, we can write:
$$\displaystyle{\sum_{i=0}^{r+1}}\displaystyle{\sum_{j\geq j_0}}\displaystyle{\sum_{k=0}^{\mu(\widetilde{P})-1}}\frac{\alpha_{i,j,k}}{\varpi(P)}\lambda^iP^{j-1}\vartheta_k\overrightarrow{\nabla}P\cdot\vec{e}_{\varpi}=
\overrightarrow{\nabla}P\cdot\left(-\lambda\overrightarrow{\nabla}\times\overrightarrow{F}+\overrightarrow{\nabla}\lambda\times\overrightarrow{F}
\right).$$
According to the exactness of the Koszul complex associated with $P,$ there exists $\overrightarrow{H}\in\mathcal A^3$ such that
\begin{equation}\label{eq2}
\displaystyle{\sum_{i=0}^{r+1}}\displaystyle{\sum_{j\geq j_0}}\displaystyle{\sum_{k=0}^{\mu(\widetilde{P})-1}}\frac{\alpha_{i,j,k}}{\varpi(P)}\lambda^iP^{j-1}\vartheta_k\vec{e}_{\varpi}=
\left(-\lambda\overrightarrow{\nabla}\times\overrightarrow{F}+\overrightarrow{\nabla}\lambda\times\overrightarrow{F}
\right)+\overrightarrow{H}\times\overrightarrow{\nabla}P.
\end{equation}
Set $\overrightarrow{H}=\lambda\overrightarrow{G}+\overrightarrow{L},$ where $\overrightarrow{G}\in\mathcal A^3$ and $\overrightarrow{L}\in\mathcal B^3=K[x,y]^3.$ Then if we consider the scalar product of left and right parts of the equation (\ref{eq2}) by $\overrightarrow{\nabla}\lambda,$ we get:
$$\displaystyle{\sum_{i=0}^{r+1}}\displaystyle{\sum_{j\geq j_0}}\displaystyle{\sum_{k=0}^{\mu(\widetilde{P})-1}}\frac{\varpi_3\alpha_{i,j,k}}{\varpi(P)}\lambda^{i+1}P^{j-1}\vartheta_k=
-\lambda\overrightarrow{\nabla}\lambda\cdot\left(\overrightarrow{\nabla}\times\overrightarrow{F}\right)-\lambda\overrightarrow{G}\cdot
\left(\overrightarrow{\nabla}\lambda\times\overrightarrow{\nabla}P\right)-\overrightarrow{L}\cdot
\left(\overrightarrow{\nabla}\lambda\times\overrightarrow{\nabla}P\right).$$
Therefore $\overrightarrow{L}\cdot
\left(\overrightarrow{\nabla}\lambda\times\overrightarrow{\nabla}P\right)\in(Z)\cap K[x,y].$ Then
\begin{equation}\label{eq3}
\overrightarrow{L}\cdot
\left(\overrightarrow{\nabla}\lambda\times\overrightarrow{\nabla}P\right)=0,
\end{equation}
and
\begin{equation}\label{eq4}
\displaystyle{\sum_{i=0}^{r+1}}\displaystyle{\sum_{j\geq j_0}}\displaystyle{\sum_{k=0}^{\mu(\widetilde{P})-1}}\frac{\varpi_3\alpha_{i,j,k}}{\varpi(P)}\lambda^{i}P^{j-1}\vartheta_k=
-\overrightarrow{\nabla}\lambda\cdot\left(\overrightarrow{\nabla}\times\overrightarrow{F}\right)-\overrightarrow{G}\cdot
\left(\overrightarrow{\nabla}\lambda\times\overrightarrow{\nabla}P\right).
\end{equation}
Suppose that $\overrightarrow{L}=(L_1, L_2, L_3)^t$ and set $\overrightarrow{L}'=(L_1, L_2)^t\in\mathcal B^2.$ Then equation (\ref{eq3}) is equivalent to $\overrightarrow{L}'\cdot\overrightarrow{\Box}\widetilde{P}=0.$ Then $\overrightarrow{L}'=L\overrightarrow{\nabla}\widetilde{P},$ $L\in\mathcal B,$ according to the exactness of the Koszul complex associated with $\widetilde{P}\in\mathcal B.$\\
Now if we compute the divergence of the left and right parts of equation (\ref{eq2}), we obtain:\\
\begin{multline}\label{eq5}
\displaystyle{\sum_{i=0}^{r+1}}\displaystyle{\sum_{j\geq j_0}}\displaystyle{\sum_{k=0}^{\mu(\widetilde{P})-1}}\frac{\alpha_{i,j,k}}{\varpi(P)}\left[i\varpi_3+(j-1)\varpi(P)+\varpi(\vartheta_k)+|\varpi|\right]
\lambda^iP^{j-1}\vartheta_k=-2\overrightarrow{\nabla}\lambda\cdot\left(\overrightarrow{\nabla}\times\overrightarrow{F}\right)\\
+\lambda\overrightarrow{\nabla}P\cdot\left(\overrightarrow{\nabla}\times\overrightarrow{G}\right)-\overrightarrow{G}\cdot
\left(\overrightarrow{\nabla}\lambda\times\overrightarrow{\nabla}P\right)+\delta^2(-\overrightarrow{\nabla}G),
\end{multline}
where $G=-\lambda^{r+1}L+L_3.$\\
From equations (\ref{eq4}) and (\ref{eq5}), we obtain:
\begin{equation}\label{eq6}
\displaystyle{\sum_{i=0}^{r+1}}\displaystyle{\sum_{j\geq j'_0}}\displaystyle{\sum_{k=0}^{\mu(\widetilde{P})-1}}\beta_{i,j,k}\lambda^{i}P^{j-1}\vartheta_k=
\delta^2(\overrightarrow{G}-\overrightarrow{\nabla}G),
\end{equation}
where $\beta_{i,j,k}=\frac{\left[j\varpi(P)+\varpi(\vartheta_k)+\varpi_1+\varpi_2+(i-1)\varpi_3\right]}{\varpi(P)}\alpha_{i,j,k}$ and $j'_0=j_0-1.$\\
\end{proof}
Then we obtain an equation of the form (\ref{eq1}), but with $j'_0<j_0.$ That is a contradiction of our hypothesis.
\begin{lem} Let $L\in\mathcal B=K[x, y].$
$\overrightarrow{\Box}L\cdot\overrightarrow{\nabla}\widetilde{P}=0$ iff $L\in K[\widetilde{P}].$
\end{lem}
\begin{proof}
Let $L$ be a weight homogeneous element of $\mathcal B$ such that $\overrightarrow{\Box}L\cdot\overrightarrow{\nabla}\widetilde{P}=0.$ Then from the exactness of the Koszul complex associated to $\widetilde{P},$ we deduce that there exists $L_1\in\mathcal B$ such that $\overrightarrow{\Box}L=L_1\overrightarrow{\nabla}\widetilde{P}.$ We have $\varpi'(L_1)<\varpi'(L),$ and $0=\overrightarrow{\Box}L\cdot\overrightarrow{\nabla}\widetilde{P}=\overrightarrow{\Box}L_1\cdot\overrightarrow{\nabla}\widetilde{P}.$ Therefore when we continue in this way, we get a sequence of elements $(L_n)_{n\in\mathbb N},$ $L=L_0,$ of $\mathcal B$ such that $\overrightarrow{\Box}L_n\cdot\overrightarrow{\nabla}\widetilde{P}=0$ and $\varpi'(L_{n+1})<\varpi'(L_n).$ Then there exists $m\in\mathbb N$ such that $\overrightarrow{\Box}L_m=0.$ Hence $L_m=c_m\in K.$ Therefore, using the exactness of the de Rham complex, there exists $c_{m-1}\in K$ such that $L_{m-1}=c_m\widetilde{P}+c_{m-1}.$ We can continue this procedure and at end we get $L\in K[\widetilde{P}].$
\end{proof}
\begin{lem}\label{lem2}
For any $s\in\mathbb N,$ there exists $F_s\in A$ such that $\widetilde{P}^s\left(\overrightarrow{\nabla}\times\overrightarrow{\nabla}P\right)={P}^s
\left(\overrightarrow{\nabla}\times\overrightarrow{\nabla}P\right)+\delta^0(F_s).$
\end{lem}
\begin{proof}
We will prove by induction on $s.$ For $s=0,$ choose $F_0=0.$ Let us suppose the result true for some $s\in\mathbb N.$ We have:
$$\begin{array}{rcl}
\delta^0(\widetilde{P}^{s+1})&=&\lambda(\overrightarrow{\nabla}\widetilde{P}^{s+1}\times\overrightarrow{\nabla}P)\\
&=&(s+1)\lambda\widetilde{P}^s(\overrightarrow{\nabla}\widetilde{P}\times\overrightarrow{\nabla}P)\\
&=&(s+1)\lambda\widetilde{P}^s\left(\overrightarrow{\nabla}({P}-\frac{1}{r+2}\lambda^{r+2})\times\overrightarrow{\nabla}P\right)\\
&=&-(s+1)\lambda^{r+2}\widetilde{P}^s(\overrightarrow{\nabla}\lambda\times\overrightarrow{\nabla}P)\\
&=&-(s+1)(r+2)(P-\widetilde{P})\widetilde{P}^s(\overrightarrow{\nabla}\lambda\times\overrightarrow{\nabla}P)\\
&=&-(s+1)(r+2)P\widetilde{P}^s(\overrightarrow{\nabla}\lambda\times\overrightarrow{\nabla}P)+
(s+1)(r+2)\widetilde{P}^{s+1}(\overrightarrow{\nabla}\lambda\times\overrightarrow{\nabla}P)\\
\end{array}$$
But $$\widetilde{P}^s\left(\overrightarrow{\nabla}\times\overrightarrow{\nabla}P\right)={P}^s
\left(\overrightarrow{\nabla}\times\overrightarrow{\nabla}P\right)+\delta^0(F_s).$$
Then
$$\widetilde{P}^{s+1}\left(\overrightarrow{\nabla}\times\overrightarrow{\nabla}P\right)={P}^{s+1}
\left(\overrightarrow{\nabla}\times\overrightarrow{\nabla}P\right)+\delta^0(F_{s+1}),$$
where $F_{s+1}=PF_s+\frac{1}{(s+1)(r+2)}\widetilde{P}^{s+1}.$
\end{proof}
\begin{theo}
$PH^1(\mathcal A, \pi)=\left\{{\begin{array}{l}
K[P](\overrightarrow{\nabla}\lambda\times\overrightarrow{\nabla}P) \ \ if\ \  \varpi(P)\neq\varpi_1+\varpi_2\\
K[P](\overrightarrow{\nabla}\lambda\times\overrightarrow{\nabla}P)\oplus K[P]\vec e_{\varpi} \ \ if \ \ \varpi(P)=\varpi_1+\varpi_2
\end{array}}\right.$
\end{theo}
\begin{proof}
Let $\overrightarrow{F}$ be a weight homogeneous element of $Ker\delta^1.$ Then
\begin{equation}\label{equa1}
\lambda\overrightarrow{\nabla}(\overrightarrow F\cdot\overrightarrow{\nabla} P)+ \left(\lambda Div(\overrightarrow F)+\overrightarrow{F}\cdot\overrightarrow{\nabla}\lambda\right)\overrightarrow{\nabla} P=0.
\end{equation}
Therefore $\lambda\overrightarrow{\nabla}(\overrightarrow F\cdot\overrightarrow{\nabla} P)\times\overrightarrow{\nabla}P=0.$ Hence $\overrightarrow F\cdot\overrightarrow{\nabla} P\in K[P].$ Set $\overrightarrow F\cdot\overrightarrow{\nabla} P=cP^{l+1},$ $l\in\mathbb N.$\\
From the exactness of the Koszul complex associated with $P,$ we deduce that there exists $\overrightarrow{G}\in\mathcal A^3$ such that
$$\overrightarrow F=\frac{c}{\varpi(P)}P^l\vec e_{\varpi}+\overrightarrow{G}\times\overrightarrow{\nabla}P.$$
Since $\delta^1(\overrightarrow{F})=0,$ we have $c\left(1-\frac{\varpi_1+\varpi_2}{\varpi(P)}\right)=\delta^2(-\overrightarrow{G}).$ Therefore, from theorem \ref{theo3}, we have $c\left(1-\frac{\varpi_1+\varpi_2}{\varpi(P)}\right)=\delta^2(-\overrightarrow{G})=0.$\\
Set $\overrightarrow{G}=\lambda\overrightarrow{H}+\overrightarrow{L},$ $\overrightarrow{H}\in\mathcal A^3$ and $\overrightarrow{L}\in\mathcal B^3.$ Then
$$\lambda^2\overrightarrow{\nabla}P\cdot(\overrightarrow{\nabla}\times\overrightarrow{H})+\lambda
\overrightarrow{\nabla}P\cdot(\overrightarrow{\nabla}\times\overrightarrow{L})+
\overrightarrow{L}\cdot(\overrightarrow{\nabla}\lambda\times\overrightarrow{\nabla}P)=0.$$
Hence $\overrightarrow{L}\cdot(\overrightarrow{\nabla}\lambda\times\overrightarrow{\nabla}P)\in(z)\cap\mathcal B=\{0\}.$ Therefore $\overrightarrow{L}\cdot(\overrightarrow{\nabla}\lambda\times\overrightarrow{\nabla}P)=\overrightarrow{L}'\cdot\overrightarrow{\Box}\widetilde{P}=0,$ where $\overrightarrow{L}'=(L_1, L_2)^t$ if $\overrightarrow{L}=(L_1, L_2, L_3)^t.$ Then there exists $L\in\mathcal B$ such that $\overrightarrow{L}'=L\overrightarrow{\nabla}\widetilde{P}.$\\
On the other hand
$$\lambda^2\overrightarrow{\nabla}P\cdot(\overrightarrow{\nabla}\times\overrightarrow{H})=-\lambda
\overrightarrow{\nabla}P\cdot(\overrightarrow{\nabla}\times\overrightarrow{L})=-\left(\frac{\partial L_2}{\partial x}-\frac{\partial L_1}{\partial y}\right)-\overrightarrow{\Box}L_3\cdot\overrightarrow{\nabla}\widetilde{P}.$$
Hence $\overrightarrow{\Box}L_3\cdot\overrightarrow{\nabla}\widetilde{P}\in\mathcal B\cap (z)=\{0\}.$ Therefore $L_3\in K[\overrightarrow{P}].$ Set $L_3=\alpha\widetilde{P}^s.$ We obtain
$$\overrightarrow{L}=L\overrightarrow{\nabla}P+(\alpha\widetilde{P}^s-\lambda^{r+1})\overrightarrow{\nabla}\lambda,$$
and
$$\overrightarrow{\nabla}P\cdot(\overrightarrow{\nabla}\times\overrightarrow{H})=\overrightarrow{\nabla}P\cdot
\left(\overrightarrow{\nabla}\times(\lambda^rL\overrightarrow{\nabla}\lambda)\right).$$
We can therefore deduce that there exist $Q_1,\ Q_2\in\mathcal A$ such that $$\overrightarrow{H}=\lambda^rL\overrightarrow{\nabla}\lambda+\overrightarrow{\nabla}Q_1+Q_2\overrightarrow{\nabla}P.$$
Finally, we obtain
$$\begin{array}{rcl}\overrightarrow{F}&=&\frac{c}{\varpi(P)}P^l\vec e_{\varpi}+\lambda(\overrightarrow{\nabla}Q_1\times\overrightarrow{\nabla}P)+
\alpha\widetilde{P}^s(\overrightarrow{\nabla}\lambda\times\overrightarrow{\nabla}P)\\
&=&\frac{c}{\varpi(P)}P^l\vec e_{\varpi}+\delta^0(Q_1)+
\alpha\widetilde{P}^s(\overrightarrow{\nabla}\lambda\times\overrightarrow{\nabla}P).
\end{array}$$
From the lemma \ref{lem2}, there exists $F_s\in\mathcal A$ such that $\alpha\widetilde{P}^s(\overrightarrow{\nabla}\lambda\times\overrightarrow{\nabla}P)=
\alpha{P}^s(\overrightarrow{\nabla}\lambda\times\overrightarrow{\nabla}P)+\delta^0(F_s).$ Then
$$\overrightarrow{F}=\frac{c}{\varpi(P)}P^l\vec e_{\varpi}+\alpha{P}^s(\overrightarrow{\nabla}\lambda\times\overrightarrow{\nabla}P)+\delta^0(Q_1+F_s),$$
with the condition that $c\left(1-\frac{\varpi_1+\varpi_2}{\varpi(P)}\right)=0.$
Since $\delta^1(\vec e_{\varpi})=(\varpi_1+\varpi_2-\varpi(P))\lambda\overrightarrow{\nabla}P,$ we can conclude that
\begin{equation}\label{equa2}
Ker\delta^1=\left\{{\begin{array}{l}
K[P](\overrightarrow{\nabla}\lambda\times\overrightarrow{\nabla}P)+Im\delta^0 \ \ if\ \  \varpi(P)\neq\varpi_1+\varpi_2\\
K[P](\overrightarrow{\nabla}\lambda\times\overrightarrow{\nabla}P)+ K[P]\vec e_{\varpi}+Im\delta^0\ \ if \ \ \varpi(P)=\varpi_1+\varpi_2.
\end{array}}\right.
\end{equation}
Since the sum in (\ref{equa2}) is direct, the result follows. Let us prove it. Consider $F,$ a weight homogeneous element of $\mathcal A,$ and $\alpha, \ \beta\in K, $ $l,\ s\in\mathbb N,$ with the condition that $\alpha=0$ if $\varpi(P)=\varpi_1+\varpi_2,$ such that
$$\alpha P^l\vec e_{\varpi}+\beta P^s\overrightarrow{\nabla}\lambda\times\overrightarrow{\nabla}P=\delta^0(F)=\lambda\overrightarrow{\nabla}\lambda\times\overrightarrow{\nabla}P.$$
Then $\alpha P^l\vec e_{\varpi}\cdot\overrightarrow{\nabla}P=0,$ and therefore $\alpha=0.$ On the other hand there exists $\widehat{P}_s\in\mathcal A$ such that $P^s=\widetilde{P}^s+\lambda^{r+2}\widehat{P}_s.$ Hence \\
$$\lambda(\overrightarrow{\nabla}\lambda\times\overrightarrow{\nabla}P)=\beta P^s(\overrightarrow{\nabla}\lambda\times\overrightarrow{\nabla}P)
=\beta \widetilde{P}^s\overrightarrow{\nabla}\lambda\times\overrightarrow{\nabla}P+\beta \lambda^{r+2}\widehat{P}_s(\overrightarrow{\nabla}\lambda\times\overrightarrow{\nabla}P).$$
We deduce that $\beta \widetilde{P}^s\overrightarrow{\nabla}\lambda\times\overrightarrow{\nabla}P\in(\lambda\mathcal A^3)\times\mathcal B^3=\{0\},$ and therefore $\beta=0.$
\end{proof}

\begin{rem}
One can observe that the Poisson homological groups of GJPS, with the assumptions of the beginning of this section,  are free and of finite type as modules over the center of the Poisson algebra $K[P].$ This is not true for their Poisson homological groups.
\end{rem}

\subsubsection*{Acknowledgements}
A Part of this work has been done when visiting Institut des Hautes Etudes Scientifiques in France. I would like to thank this institute for the invitation and for good working conditions. I thank the referee of this paper for his comments and suggestions which are incoporated in this revised version. I have been partially financed by ``Fond National de Recherche (Luxembourg)''.

\end{document}